\date{5 February 2013. Revised: 29 August 2013.}

\documentclass[12pt]{amsart}
\usepackage{latexsym,amsmath,amsfonts,amscd,amssymb,mathrsfs}

\newcommand{\surj}{\twoheadrightarrow}
\newcommand{\inc}{\hookrightarrow}

\newcommand{\x}{\times}
\newcommand{\ox}{\otimes}

\newcommand{\too}{\longrightarrow}
\newcommand{\la}{\langle}
\newcommand{\ra}{\rangle}

\newcommand{\Lie}[1]{\operatorname{#1}}
\newcommand{\lie}[1]{\operatorname{\mathfrak{#1}}}
\newcommand{\GL}{\Lie{GL}}

\newcommand{\SO}{\Lie{SO}}

\newcommand{\Spin}{\Lie{Spin}}
\newcommand{\SU}{\Lie{SU}}
\newcommand{\su}{\lie{su}}
\newcommand{\U}{\Lie{U}}

\DeclareMathAlphabet{\mathpzc}{OT1}{pzc}{m}{it}

\DeclareMathOperator{\rk}{rk\,}          

\DeclareMathOperator{\End}{End\,}           
\DeclareMathOperator{\id}{Id\,}             
\DeclareMathOperator{\tr}{Tr\,}             
\DeclareMathOperator{\Hol}{Hol}       
\DeclareMathOperator{\Aut}{Aut\,}           

\newcommand{\vol}{\mathrm{vol}}

\newcommand{\Sym}{\text{Sym}}

\renewcommand{\Re}{\mathrm{Re}}
\renewcommand{\Im}{\mathrm{Im}}

\newcommand{\cA}{{\mathcal A}}

\newcommand{\cD}{{\mathcal D}}
\newcommand{\cE}{{\mathcal E}}
\newcommand{\cF}{{\mathcal F}}
\newcommand{\cG}{{\mathcal G}}
\newcommand{\cH}{{\mathcal H}}

\newcommand{\cL}{{\mathcal L}}
\newcommand{\cO}{{\mathcal O}}

\newcommand{\cQ}{{\mathcal Q}}
\newcommand{\cR}{{\mathcal R}}

\newcommand{\CC}{{\Bbb C}}

\newcommand{\QQ}{{\Bbb Q}}
\newcommand{\RR}{{\Bbb R}}

\newcommand{\ZZ}{{\Bbb Z}}

\renewcommand{\a}{\alpha}

\theoremstyle{plain}
\begingroup
\newtheorem{theorem}{Theorem}
\newtheorem{corollary}[theorem]{Corollary}
\newtheorem{lemma}[theorem]{Lemma}
\newtheorem{proposition}[theorem]{Proposition}

\endgroup

\theoremstyle{definition}
\newtheorem{definition}[theorem]{Definition}

\theoremstyle{remark}
\newtheorem{remark}[theorem]{Remark}

\title[$\Spin(7)$-instantons and stable bundles on complex $4$-tori]
{$\Spin(7)$-instantons, stable bundles and the Bogomolov inequality for complex $4$-tori}


\author{Vicente Mu\~noz}
  \address{Facultad de Matem\'{a}ticas \\ Universidad Complutense
  de Madrid \\ Plaza Ciencias 3
  \\ 28040 Madrid \\ Spain}
  \email{vicente.munoz@mat.ucm.es}

\thanks{Partially supported through grant MICINN (Spain) MTM2010-17389}
\begin{document}

\begin{abstract}
 Using gauge theory for $\Spin(7)$ manifolds of dimension $8$, we develop a procedure, 
 called $\Spin$-rotation, which transforms a
 (stable) holomorphic structure on a vector bundle over a complex torus of
 dimension $4$ into a new holomorphic structure over a different
 complex torus. We show non-trivial examples of this procedure by rotating a 
 decomposable Weil abelian variety into a non-decomposable one.
 As a byproduct, we obtain a Bogomolov type inequality, which gives restrictions
 for the existence of stable bundles on an abelian variety of dimension $4$, and
 show examples in which this is stronger than the usual Bogomolov inequality.
\end{abstract}

\maketitle

\section{Introduction} \label{sec:intro}

Let $M$ be a smooth four-dimensional compact oriented manifold, and
let $E\to M$ be a hermitian complex vector bundle. 
If $M$ is endowed with a K\"ahler structure, then 
the Hitchin-Kobayashi correspondence gives a one-to-one correspondence between 
(poly)stable holomorphic structures on $E$ and anti-self-dual connections (instantons)
on $E$. A K\"ahler structure on $M$ is basically a Riemannian metric 
with holonomy in $\U(2)\subset \SO(4)$. Under this inclusion, 
the Hermite-Einstein equation is translated into the anti-self-dual equation.
Therefore, if we know of the existence of anti-self-dual connections for $E$
(e.g.\ if the Donaldson invariant corresponding to $E$ is non-zero, which
guarantees the non-emptiness of the moduli space of anti-self-dual connections),
then for \emph{any} K\"ahler structure $\omega$ on $M$ compatible with the given metric, $E$ admits an 
structure of polystable holomorphic bundle. In particular, if $(M,\omega)$ is a
projective surface, then $c_1(E)$ is an algebraic class. However,
this does not say anything new about the algebraicity of Hodge classes for
complex projective surfaces, since the Hodge Conjecture is known for $(1,1)$-classes
(cf.\ Lefschetz theorem on $(1,1)$-classes \cite{Lefschetz}).

In \cite{Toma}, M.\ Toma uses the above idea to find stable bundles
with given Chern classes on a K\"ahler $4$-torus $Y$. By what we said above, the condition to have
a (poly)stable holomorphic structure is a condition on the Riemmanian structure of $Y$.
So one may vary the inclusion $\U(2)\subset \SO(4)$, hence changing the K\"ahler structure
on $Y$ without changing the Riemannian metric. Under these circumstances, 
the existence of a stable holomorphic structure
on a bundle $E$ for one K\"ahler structure $\omega$ on $Y$ implies the same for another 
K\"ahler structure $\omega'$. A similar approach is used in \cite{Verbitsky} for
studying stable bundles on hyperk\"ahler and hypercomplex manifolds.

\smallskip

In eight dimensions, we can try to follow the same scheme, by considering
the inclusion $\SU(4)\subset \Spin(7)$. The proposal to develop 
gauge theoretic methods in $6$, $7$ and $8$ dimensions appears in
the seminal paper of S.\ Donaldson and R.\ Thomas \cite{Donaldson-Thomas}.
In dimension $8$, it is natural to consider a (compact, oriented) manifold $M$ with
a $\Spin(7)$-structure (that is, endowed with a Riemannian metric whose holonomy lies in 
$\Spin(7)\subset \SO(8)$). For such $M$, the notion of $\Spin(7)$-instanton
is defined in \cite{Donaldson-Thomas}. Ideally, this should give rise to 
moduli spaces of $\Spin(7)$-instantons and associated invariants (in the vein
of the Donaldson invariants in four dimensions).
Basic properties of the linearized $\Spin(7)$-instanton equation 
are considered in \cite{Reyes}. Construction of non-trivial examples of
$\Spin(7)$-instantons are given in the D. Phil.\ thesis of C.\ Lewis 
\cite{Lewis}. The deep analysis of the bubbling properties of
$\Spin(7)$-instantons is undertaken by G.\ Tian in \cite{Tian}. 
However, the transversality issues related to the construction of the
moduli spaces of $\Spin(7)$-instantons seem to be out of reach at present.

An $8$-dimensional compact Riemannian manifold $M$ with holonomy in $\SU(4)$ is
a manifold with a Calabi-Yau structure (a K\"ahler manifold $(M,\omega)$ 
endowed with a global holomorphic everywhere non-zero $(4,0)$-form $\theta$). 
Let $E\to M$ be an $\SU(r)$-vector bundle.
A $\Spin(7)$-instanton on $E$ is equivalent to a polystable holomorphic
structure under the obvious topological condition $c_2(E)\in H^{2,2}(M)$. 
This result already appears in \cite{Lewis}. 
G.\ Tian points out \cite{Tian} that this could be used to prove the
Hodge conjecture if there is a way to prove the existence of $\Spin(7)$-instantons
on $E$. The observation appears very explicitly in \cite{Ramadas}, where
an (unsuccessful) attempt to construct $\Spin(7)$-instantons for some 
abelian varieties of Weil type, for which the Hodge Conjecture is yet unknown, 
is proposed.

\smallskip

In this paper we study $\Spin(7)$-instantons for K\"ahler complex tori
of dimension $4$. We consider $\SU(4)$-structures giving rise to 
the same (flat) $\Spin(7)$-structure on an $8$-torus $X=\RR^8/\Lambda$,
determining under which circumstances the existence of a (poly)stable
holomorphic structure on a vector bundle $E\to X$ for $(X,\omega,\theta)$ produces
a (poly)stable holomorphic structure on $E$ for $(X,\omega',\theta')$
(cf.\ Theorem \ref{thm:rotation}). We call 
the process of going from $(\omega,\theta)$ to $(\omega',\theta')$ a 
\emph{$\Spin$-rotation}, as it is given by conjugating the inclusion
$\SU(4)\inc \Spin(7)$ by an element in $\Spin(7)$. 
The reason that this is realisable hinges on
the fact that we do not try to prove the non-emptiness of 
the moduli spaces of $\Spin(7)$-instantons
by relating them for different metrics, we just use that the $\Spin(7)$-structure of
$X$ is fixed, so 
a given $\Spin(7)$-instanton keeps satisfying the $\Spin(7)$-instanton equation
while performing a $\Spin$-rotation.

\smallskip

The Hodge Conjecture predicts that Hodge classes (rational $(p,p)$-classes)
on a projective complex $n$-manifold are represented by (rational) algebraic
cycles (see \cite{survey} for a nice account on the subject). The Hodge Conjecture is known
to be true for dimension $n\leq 3$ and for $p=1,n-1$. The first instance
in which it is still open is $n=4$, $p=2$. A.\ Weil proposed \cite{Weil} a
family of abelian varieties of dimension $4$ (based on an example of
Mumford \cite{Mumford}) to be studied as possible counterexamples to the
Hodge Conjecture. These abelian varieties have complex multiplication 
by $L=\ZZ[\sqrt{-d}]$, $d>0$ a square-free integer,
and a subspace of Hodge $(2,2)$-classes (now known as \emph{Weil classes})
which are not products of $(1,1)$-classes, and therefore potentially 
non-algebraic. Surprisingly, for the case of complex multiplication 
by $\ZZ[\sqrt{-d}]$ with $d=1,3$, C.\ Schoen \cite{Schoen} proved that 
the generic Weil abelian variety satisfies the Hodge Conjecture, by
constructing the algebraic cycles 
representing the Weil classes (the method of proof is algebro-geometric
and not extendable to other values of $d$). 
See \cite{Moonen-Zarhin} for an account on the state of the
Hodge Conjecture for abelian four-folds.

One possible route to construct algebraic cycles for a Weil abelian variety
is to construct stable bundles with given Chern classes. The $\Spin$-rotations
are a new method to construct stable bundles on complex tori, and have the nice
feature of being a highly non-algebro-geometric proceduce. 
We show an example where $\Spin$-rotation is performed starting with 
an abelian variety of Weil type which is a product of two abelian surfaces.
The resulting $\Spin$-rotated torus is another Weil abelian variety which
turns out to be non-decomposable (see Section \ref{sec:rotating-Weil}), very different from 
the initial one from the algebro-geometric and arithmetic points of view. 
This is a deep indication that the $\Spin$-rotation have a very non-trivial
effect.

\smallskip

On the negative side, C.\ Voisin \cite{Voisin} gave examples of K\"ahler
$4$-tori which are not algebraic for which the Hodge Conjecture does not
hold. These tori are of Weil type, and the non-existence of complex cycles 
is proved via the non-existence of (stable) holomorphic structures on 
vector bundles, which in turn hinges on the Bogomolov inequality.
The $\Spin$-rotations give an explanation to the phenomenon in \cite{Voisin}.
For a bundle to be stable, it has to be positive enough, so the 
examples of \cite{Voisin} cannot be $\Spin$-rotated to abelian varieties.

Actually, we get 
non-trivial restrictions for the existence of stable bundles
with given Chern classes, producing a Bogomolov type inequality:

\noindent \textbf{Theorem.}
\emph{Let $(X,\omega)$ be an abelian four-fold. Let $\beta_0\in H^{2,2}_{prim}(X)$ be a rational
 primitive $(2,2)$-class, and define 
 $$
  k_m(\beta_0)=\frac14\, \max\{ \frac{\beta_0\cup c\cup \bar{c}}{\vol(X)} \, | \, c\in H^{2,0}(X), |c|_\omega = 1 \}.
  $$
 Then, for any $\omega$-polystable bundle $E$, such that the class $\beta(E)= c_2(E)-\frac{r-1}{2r}c_1(E)^2$
 has component in $H^{2,2}_{prim}(X)$ equal to $\beta_0$, we have
   $$
    \beta(E) \cup [\omega]^2 \geq k_m(\beta_0)\,  [\omega]^4\, .
    $$}

%

\smallskip

\noindent \textbf{Acknowledgements:} I am indebted to many people for
very useful conversations, suggestions and encouragement. In chronological order,
I have to mention Gang Tian, Ludmil Katzarkov, Simon Donaldson, Richard Thomas, Claire Voisin,
Ignasi Mundet, Francisco Presas, Javier Fern\'andez de Bobadilla,
Javier Mart\'{\i}nez, Luis Sol\'a, Matei Toma, Philippe Eyssidieux, Adrian Langer and Misha Verbitsky. 
Special thanks to Dominic Joyce for providing me with a copy of \cite{Lewis}. 
Also I am grateful to the referee for useful comments which helped to improve the exposition of the paper.

\section{$\Spin(7)$ geometry} \label{sec:spin7}

\subsection{The group $\Spin(7)$}\label{subsec:group-spin7}
Consider $\RR^8$ with coordinates $(x_1,\ldots, x_8)$, and the
following $4$-form:
  \begin{equation}\label{eqn:1}
  \begin{aligned}
   \Omega_0 = & \ dx_{1234}+ dx_{1256}+ dx_{1278} +dx_{1357} -dx_{1368}
   -dx_{1458}-dx_{1467} \\ &  -dx_{2358} -dx_{2367} -dx_{2457} +
   dx_{2468} + dx_{3456} +dx_{3478} +dx_{5678}\, ,
   \end{aligned}
  \end{equation}
where we abbreviate $dx_{ij\ldots k}= dx_i\wedge dx_j \wedge\ldots
\wedge dx_k$.
The group $\Spin(7)$ is the subgroup of $\GL_+(8,\RR)$ which
leaves invariant $\Omega_0$. 
An $\Spin(7)$-structure on an oriented vector space of
dimension $8$ is the choice of a $4$-form $\Omega$ which can be
written as (\ref{eqn:1}) in a suitable oriented frame. 

Given $\Omega$, there is a well-defined scalar product $g$. This
holds because of the inclusion 
$\Spin(7)\subset \SO(8)$. The Hodge star
operator $*$ associated to $g$ makes $\Omega$ self-dual,
$*\, \Omega= \Omega$. The standard volume form is $\vol=dx_{12\ldots 8}$.
Then $\Omega\wedge \Omega=14 \, \vol$.

We shall use some representation theory of $\Spin(7)$.
The standard representation is $V=\RR^8$ with the action given by
$\Spin(7) \inc \GL(8,\RR)$. There is a $7$-dimensional spin
representation $S=\RR^7$ given by the double cover $\Spin(7) \to
\SO(7)\subset \GL(7,\RR)$. Let $\bigwedge^i =\bigwedge^i V$,
$1\leq i\leq 8$. Then we have the following decomposition of the
representations $\bigwedge^i$ into irreducible factors:
 $$
 \begin{aligned}
 \bigwedge\nolimits^1 = & \, \bigwedge\nolimits^1_8 \\
 \bigwedge\nolimits^2 = & \, \bigwedge\nolimits^2_7 \oplus \bigwedge\nolimits^2_{21}\\
 \bigwedge\nolimits^3 = & \, \bigwedge\nolimits^3_8 \oplus \bigwedge\nolimits^3_{48} \\
 \bigwedge\nolimits^4 = & \, \bigwedge\nolimits^4_+ \oplus \bigwedge\nolimits^4_-,
  \qquad  \text{where } 
 \bigwedge\nolimits^4_+=\bigwedge\nolimits^4_1\oplus \bigwedge\nolimits^4_7\oplus
 \bigwedge\nolimits^4_{27},
  \quad  \bigwedge\nolimits^4_-=\bigwedge\nolimits^4_{35} \, ,
 \end{aligned}
 $$
where we write $\bigwedge^i_j$ for an irreducible
sub-representation of $\bigwedge^i$ of dimension $j$.

These are characterized as follows. 
$\bigwedge^2_7$ is isomorphic to the spin representation $S$
and consists of those $\alpha \in \bigwedge^2$ such that 
$*(\Omega\wedge \alpha)=3\alpha$ (equivalently, 
$\alpha\in \bigwedge^2_7$ if and only if
$\Omega\wedge\alpha \wedge \alpha =3|\alpha|^2$). The summand
$\bigwedge^2_{21}$ is characterised by the equation
$*(\Omega\wedge \alpha)=-  \alpha$.

The summand 
$\bigwedge^3_8$ consists of elements 
of the form $*(\Omega \wedge \phi)$, $\phi\in \bigwedge^1$. And 
$\bigwedge^3_{48}$ is its orthogonal complement, i.e.\ the kernel
of $\wedge \Omega: \bigwedge^3 \to \bigwedge^7$.

The Hodge operator $*$ acts as $\pm 1$ on $\bigwedge^4_\pm$. The summand
$\bigwedge^4_1$ is generated by $\Omega$. 
The summand $\bigwedge^4_7$ is isomorphic to $S$, and it is the image
under the chain of maps $S
\subset \bigwedge^2 V \subset V\otimes V \cong
\bigwedge^1_8\otimes \bigwedge^3_8 \stackrel{\wedge}{\too}
\bigwedge^4$. The piece $\bigwedge^4_{27} \cong \Sym_0^2 \, S$, the 
traceless part of the symmetric product, and it appears as the image under
$\Sym_0^2 \, S \subset S\otimes S \cong \bigwedge^2_7\otimes \bigwedge^2_7
\stackrel{\wedge}{\too} \bigwedge^4$. Finally,
$\bigwedge^4_{35}\cong \bigwedge^3 S$ under $\bigwedge^3 S\subset
S \otimes \bigwedge^2 S \cong \bigwedge^2_7\otimes \bigwedge^2_{21}
\stackrel{\wedge}{\too} \bigwedge^4$.

By the discussion above, 
the wedge product gives an isomorphism 
$\Sym^2 \bigwedge^2_7 \cong \bigwedge_{27}^4\oplus \bigwedge_1^4$.
Also the wedge product goes as $\bigwedge^2_7 \bigotimes
\bigwedge^2_{21} \too \bigwedge^4_7 \oplus \bigwedge^4_{35}$, surjectively.

\medskip

Let $M$ be a compact smooth oriented manifold of dimension $8$. A
$\Spin(7)$-structure on $M$ is the choice of a $4$-form $\Omega\in
\Omega^4(M)$ such that:
 \begin{itemize}
\item $\Omega_p$ is a $\Spin(7)$-structure on
$T_pM$ for each $p\in M$. The $\Spin(7)$-structure induces a
Riemannian metric $g$.
\item $\nabla \Omega=0$,
where $\nabla$ is the Levi-Civita connection associated to $g$.
This is equivalent to $\Omega$ being closed and co-closed \cite{Salamon},
$d\Omega=d*\Omega=0$. But as $\Omega$ is self-dual, it is
equivalent to $d\Omega=0$. 
 \end{itemize}
Note that $[\Omega] \in H^4(M)$ is non-trivial. 

If $\Omega$ is a $\Spin(7)$-structure, we say that $(M,\Omega)$ is a $\Spin(7)$-manifold.
Note that this is also equivalent to having a Riemannian metric $g$
on $M$ such that its (restricted) holonomy group satisfies $\Hol_g\subset
\Spin(7)$: choose $\Omega_{p_0}$ invariant by $\Spin(7)$ at some
$p_0\in M$, and parallel transport it by the connection to get a
global $\Omega$. However, note that we shall understand that a
$\Spin(7)$-manifold is a Riemannian manifold with holonomy
contained in $\Spin(7)$ \emph{together with} a chosen $4$-form as above.

For a $\Spin(7)$-manifold, the Laplacian on forms leaves
invariant the bundles $\bigwedge^i_j (TM)$. So it induces a
decomposition on harmonic forms as $\cH^i(M)= \bigoplus
\cH^i_j(M)$, accordingly. However, note that $\cH^i_j(M)$ are not
of dimension $j$, and they are not $\Spin(7)$-representations
 (cf.\ \cite{Joyce}).

\subsection{Relationship between $\SU(4)$ and $\Spin(7)$} \label{subsec:group-su4}

Let $V$ be a complex vector space of dimension $4$. This can be
understood as a real vector space of dimension $8$, $V= \RR^8$
together with an endomorphism $J:V\to V$, such that $J^2=-\id$.
This corresponds to the inclusion $\GL(4,\CC)\subset
\GL_+(8,\RR)$. An $\SU(4)$-structure is the same as the choice of:
 \begin{itemize}
 \item an hermitian metric $h$ on $V$. Correspondingly, $g=\Re (h)$
 is a scalar product compatible with $J$ (i.e.,\ $g(Jx,Jy)=g(x,y)$),
 and $\omega(x,y)=g(x, J y)$ is a $(1,1)$-form, also compatible
 with $J$. This corresponds to $\U(4)\subset \SO(8)$.
 \item a non-zero $(4,0)$-form $\theta$. This produces a
 trivialization of $\bigwedge^{4,0} V$. 
 \end{itemize}

There is a complex frame $(z_1,\ldots, z_4)=(x_1 + i x_2, \ldots, x_7+ix_8)$ such that
 $$
 \begin{aligned}
 \omega =& \,  \frac{i}2 ( dz_1 \wedge d\bar{z}_1 + \ldots +dz_4 \wedge
 d\bar{z}_4) = dx_{12}+ dx_{34}+ dx_{56}+dx_{78} \, ,\\
 \theta =& \, dz_1\wedge \ldots \wedge dz_4 \,.
 \end{aligned}
 $$
The group $\U(4)$ is the subgroup of $\GL_+(8,\RR)$ 
which leaves $\omega, J$ fixed (alternatively, $\omega, g$ fixed). The group
$\SU(4)$ leaves $\omega, J$ and $\theta$ fixed. Note that
$\U(4)\subset \SO(8)$. 

There is an inclusion $\SU(4)\subset \Spin(7)$, so an $\SU(4)$-structure induces a
$\Spin(7)$-structure. This holds by choosing the following real
$4$-form:
  \begin{equation}\label{eqn:Omega}
  \Omega= \frac12 \omega\wedge\omega + \Re (\theta).
  \end{equation}
It is easy to see that if we write $z_1=x_1+ix_2, \ldots,
z_4=x_7+ix_8$, then (\ref{eqn:Omega}) equals (\ref{eqn:1}).

Note that $\theta\wedge \bar{\theta}=16 \, \vol$,
$\Re(\theta) \wedge\Re(\theta) = \frac12 \theta \wedge \bar\theta = 8\, \vol$
and $\omega^4 = 24 \, \vol$. In particular, 
as $*\theta=\bar\theta$, we have that $|\theta|=4$, 
$|\Re(\theta)|= 2 \sqrt{2}$ and $|\omega|=2$.

For $\a\in \bigwedge^2_7$, we have $\Omega \wedge\a \wedge \a= \la 
* (\Omega \wedge\a) ,\a \ra =\la 3 \a, \a \ra = 3 |\a|^2$. As
 $$
 \Omega \wedge\omega \wedge \omega=\frac12 \omega^4= 12 \, \vol = 3 | \omega|^2,
 $$
we have that $\omega \in \bigwedge^2_7$.
Fixing a $\Spin(7)$-structure, the compatible $\SU(4)$-structures
(those inducing the given $\Spin(7)$-structure) are parametrized
by the homogeneous space
  $$
  P=\Spin(7)/\SU(4)\, .
  $$
Here the group $\Spin(7)$ acts on the subgroups $G < \Spin(7)$ which
are isomorphic to $\SU(4)$, by conjugation.

\begin{lemma} \label{lem:1}
  $P$ is diffeomorphic to $S(\bigwedge^2_7)$, which is a 
  $6$-sphere.
\end{lemma}

\begin{proof}
  The elements of $P$ can be understood as pairs $(\omega,\theta)$ satisfying
  (\ref{eqn:Omega}), where $\Omega$ is fixed. Let $P'=S(\bigwedge^2_7)$, the sphere of
  radius $2$. We
  consider 
  the map $\Phi:P\to P'=S(\bigwedge^2_7)$, $\Phi(\omega,\theta)= \omega$.
  This is well-defined, since $\omega \in\bigwedge^2_7$ and $|\omega|= 2$.
  The map is clearly $\Spin(7)$-equivariant, and if $\phi\in \Spin(7)$
  leaves $\omega$ fixed, then it fixes $\Re(\theta)$ as well. Therefore it fixes
  $\theta$ which is the $(4,0)$-component of $\Re(\theta)$ ($J$ being fixed
  as well). So $\Phi$ is injective. 

  The dimension of $P$ is $\dim P=\dim \Spin(7)-\dim \SU(4)=21-15=6$,
  so $P$ and $P'$ are both smooth $6$-dimensional manifolds. Thus
  $\Phi$
  is regular, hence a local diffeomorphism, and so a covering. This implies that it is
  a diffeomorphism by simply-connectivity of the sphere.
\end{proof}

Lemma \ref{lem:1} says that 
given $\omega\in S(\bigwedge^2_7)$, with $|\omega|=2$, then
we have a well-defined $\SU(4)$-structure, where
$\Re(\theta)=\Omega-\frac12 \omega^2$, the complex structure $J$ is defined by
$g(x,y)=\omega(Jx,y)$, and 
$\Im(\theta)(u_1,u_2,u_3,u_4)=\Re(\theta)(Ju_1,u_2,u_3,u_4)$.

\medskip

The irreducible real representations of $\SU(4)$ are as follows.
Let $V$ be the standard representation of $\SU(4)$, that is
$V= \RR^8$ with the action given by $\SU(4) \subset \GL(4,\CC)
\inc \GL(8,\RR)$. Let $\bigwedge^k_\CC =(\bigwedge^k V)\ox \CC$,
$1\leq k\leq 8$. The action of $J$ gives a decomposition
  $$
  \bigwedge\nolimits^k_\CC=\bigoplus_{i+j=k \atop
  0\leq i,j\leq 4} \bigwedge\nolimits^{i,j} V\, .
  $$
We have the isomorphism 
$\bigwedge^{i,j} V\cong \overline{\bigwedge^{j,i} V}$. For $i < j$, we denote
  $$
  \bigtriangleup^{i,j}=\Re(\bigwedge\nolimits^{i,j} \oplus
  \bigwedge\nolimits^{j,i}),
  $$
which is a real representation, whose complexification is
$\bigwedge\nolimits^{i,j} \oplus \bigwedge\nolimits^{j,i}$. It is
irreducible and of (real) dimension $2\binom{4}{i}\binom{4}{j}$.
Its elements are of the form $\frac12(\alpha+\bar\alpha)$,
$\alpha \in \bigwedge\nolimits^{i,j}$. For given $i$, we
denote
  $$
  \bigtriangleup^{i,i}=\Re(\bigwedge\nolimits^{i,i}),
  $$
which is irreducible of (real) dimension $\binom{4}{i}^2$.

The choice of $\omega \in
\bigwedge^{1,1}$ gives a further decomposition. For $i+j\leq 4$,
  $$
  \bigwedge\nolimits^{i,j} =  \bigoplus_{r=m}^M
  \bigwedge\nolimits^{i-r,j-r}_{prim}\cdot \, \omega^r\, ,
  $$
where $m=\max\{0,i+j-4\}$, $M =\min\{i,j \}$, and
the primitive components are defined by
  $$
   \bigwedge\nolimits^{a,b}_{prim}
   =\ker(\omega^{5-(a+b)}:\bigwedge\nolimits^{a,b} \to
   \bigwedge\nolimits^{5-b,5-a}), 
  $$
for $a+b\leq 4$.
As $\omega$ is a real form, we have also a decomposition of the
real representations
  $$
  \bigtriangleup^{i,j} =  \bigoplus_{r=m}^M
  \bigtriangleup^{i-r,j-r}_{prim}\cdot \omega^r\, ,
  $$
where $\bigtriangleup^{a,b}_{prim}$ is defined in the obvious manner.

\smallskip

The element $\theta$ gives 
an extra decomposition. There is a complex linear map \cite{Donaldson-Thomas}
 $$
 L: \bigwedge\nolimits^{2,0} \too \bigwedge\nolimits^{0,2}\, ,
 $$
determined uniquely by
 $$
 \alpha \wedge \overline{L(\alpha)}=  \frac14 |\alpha|^2 \, \theta.
 $$
As $|\theta|=4$, 
the map $L$ is isometric.
Extend $L$ to $L:\bigwedge^{0,2}\to \bigwedge^{2,0}$, via $L(\bar\a)=\overline{L(\a)}$.
So $\alpha \wedge \overline{L(\alpha)}=  \frac14 |\alpha|^2 \bar\theta$, for $\a \in \bigwedge^{0,2}$.
Thus $L$ gives an endomorphism of $\bigtriangleup^{2,0}=\Re(\bigwedge^{2,0} \oplus
\bigwedge^{0,2})$.
It is easy to see that $L^2=\id$, so that $L$
gives a decomposition of $\bigtriangleup^{2,0}$ into $(\pm 1)$-eigenspaces. Write
 \begin{equation}\label{eqn:A+}
 \bigtriangleup^{2,0}=A_+\oplus A_- \, .
 \end{equation}
Both $A_+,A_-$ are real representations of dimension $6$.

Finally, $\bigtriangleup^{4,0}$ is of dimension $2$, and it decomposes
as $\la \Re (\theta)\ra \oplus \la \Im (\theta)\ra$.
To sum up, we have the following decomposition of the representations
$\bigwedge^i$ into irreducible summands:
 $$
 \begin{aligned}
 \bigwedge\nolimits^1 = & \, \bigtriangleup^{1,0}, \\
 \bigwedge\nolimits^2 = & \,  A_+ \oplus A_- \oplus  \bigtriangleup^{1,1}_{prim}\oplus\la \omega\ra, \\
 \bigwedge\nolimits^3 = & \, \bigtriangleup^{3,0} \oplus \bigtriangleup^{2,1}_{prim} \oplus \bigtriangleup^{1,0}\omega, \\
 \bigwedge\nolimits^4 = & \, \la \Re (\theta)\ra \oplus \la \Im (\theta)\ra \oplus
          \bigtriangleup^{3,1}_{prim}\oplus \bigtriangleup^{2,2}_{prim} \oplus A_+\omega \oplus A_-\omega
          \oplus\bigtriangleup^{1,1}_{prim} \omega \oplus \la \omega^2\ra .
 \end{aligned}
 $$

The relationship of the irreducible
representations of $\SU(4)$ and $\Spin(7)$ is given by the following result.

\begin{proposition} \label{prop:lem:2}
Let $V$ be an $8$-dimensional vector space with an $\SU(4)$-structure, and
consider the induced $\Spin(7)$-structure. Then the irreducible
$\Spin(7)$-representations decompose into $\SU(4)$-representations as follows:
 $$
 \begin{aligned}
 &\bigwedge\nolimits^2_7 =  \la \omega \ra \oplus A_+\, ,\\
 &\bigwedge\nolimits^2_{21} =  \bigtriangleup^{1,1}_{prim} \oplus A_-\, ,\\
 &\bigwedge\nolimits^3_8 =  \bigtriangleup^{3,0}\oplus \bigtriangleup^{1,0}\omega\,,\\
 &\bigwedge\nolimits^3_{48} =  \bigtriangleup^{2,1}_{prim} \, ,\\
 &\bigwedge\nolimits^4_1 =  \la \Omega\ra = \la \frac12\omega^2 +\Re(\theta)\ra \,, \\
 &\bigwedge\nolimits^4_7 =  A_-\omega \oplus \la \Im(\theta) \ra \,,\\
 &\bigwedge\nolimits^4_{27} =  A_+\omega \oplus \bigtriangleup^{2,2}_{prim}\oplus
     \la\omega^2 -\frac32 \Re(\theta)\ra \,,\\
 &\bigwedge\nolimits^4_{35}=  \bigtriangleup^{1,3}_{prim} \oplus \bigtriangleup^{1,1}_{prim}\omega\,.
 \end{aligned}
 $$
\end{proposition}

\begin{proof}
We already know that $\omega \in \bigwedge^2_7$. Let $a\in A_\pm$. 
Then $a= \frac12 (\a+\bar\a)$ and $L(a)=\pm a$. So
$L(\a)=\pm \bar\a$. Therefore
 \begin{equation*}
 \left\{ \begin{array}{l}
 \a\wedge \bar\a \wedge \omega^2= 2 |\a|^2 \vol \\
 \a\wedge\a \wedge \frac14 \bar\theta = \pm \a \wedge \overline{L(\a)} 
\wedge \frac14 \bar\theta = \pm |\a|^2 \, \frac14 \theta \wedge \frac14 \bar\theta 
=\pm |\a|^2 \vol
 \end{array} \right. 
 \end{equation*}
the first equality being always true for $(2,0)$-forms.
So, using that $a=\frac12 (\alpha+\bar\alpha)$, 
 \begin{equation*}
 \left\{ \begin{array}{l}
 a \wedge a \wedge \frac12 \omega^2 = \frac12 \alpha \wedge \bar\alpha 
\wedge \frac12 \omega^2 = \frac12 |\alpha|^2 \vol=
|a|^2 \vol \\
 a \wedge a \wedge \Re(\theta)= \frac18 (\a\wedge\a \wedge \bar\theta +
\bar\a\wedge\bar\a \wedge \theta)= \pm |\a|^2 \vol =
\pm 2 |a|^2 \vol 
 \end{array} \right. 
 \end{equation*}
We compute
 \begin{eqnarray*}
  a\wedge a \wedge\Omega &=& 
 a \wedge a \wedge \left( \frac12 \omega^2 + \Re(\theta) \right) \\
  &=& (|a|^2 \pm 2 |a|^2) \vol\, .
\end{eqnarray*}
So, for $a\in A_+$, $a\wedge a \wedge\Omega =3|a|^2 \vol$. Hence $A_+
\subset \bigwedge^2_7$. 

For $a \in A_-$, we have $a\wedge a \wedge\Omega =- |a|^2 \vol$,
proving that $A_-\subset \bigwedge^2_{21}$. The decomposition of $\bigwedge^2_{21}$
now follows by dimensionality reasons.

The decompositions in the third and fourth lines also follow by dimensionality
reasons.

The fifth line is by definition. Now the dimensions of
$A_-\omega$, $A_+\omega$, $\bigtriangleup^{2,2}_{prim}$, $\bigtriangleup^{1,3}_{prim}$
and $\bigtriangleup^{1,1}_{prim}\omega$ are $6,6, 20, 20$ and $15$, respectively.
Therefore the dimensions of the lines sixth to eighth should be 
$1+6, 1+20+6, 15+20$. This implies that $\bigtriangleup^{1,1}_{prim}\omega \subset 
\bigwedge^4_{35}$. Now recall that 
$\bigwedge^4_{35} =\bigwedge^4_-$. So to see that
$\bigtriangleup^{1,3}_{prim}\subset \bigwedge^4_{35}$, it is
enough to see that for $\alpha\in \bigtriangleup^{1,3}_{prim}$ we
have $\alpha\wedge \alpha =-|\alpha|^2\vol$. Taking as example
$\alpha=\Re(dz_1\wedge d\bar{z}_2\wedge d\bar{z}_3\wedge
d\bar{z}_4)$, we check this. This proves the eighth line.

For the seventh line, recall that image of $\Sym^2 (\bigwedge^2_7)$ 
(under the wedge product) is
$\bigwedge^4_{27}\oplus \bigwedge^4_{1}$. This image contains
$A_+ \omega$ and $\omega^2$. As it also contains $\Omega$, we have
$\Re(\theta)\in \bigwedge^4_{27}\oplus \bigwedge^4_{1}$. 
Since $\omega^2-\frac32 \Re(\theta) \perp
\frac12\omega^2 +\Re(\theta)$, it must be
$\omega^2-\frac32 \Re(\theta) \in \bigwedge_{27}^4$. 
Also $A_+\omega \subset \bigwedge_{27}^4$. The result follows.

For the sixth line, note that it must be $A_-\omega \subset \bigwedge^4_7$.
The extra term must be $\Im(\theta)$, since
$\Im(\theta)\perp \Re(\theta),\omega^2$.
\end{proof}

\begin{remark}\label{rem:aspa}
As we said in Subsection \ref{subsec:group-spin7},
there is an isomorphism $\Sym^2 \bigwedge^2_7 \cong \bigwedge_{27}^4\oplus \bigwedge_1^4$.
Under it, we see that  
\begin{equation}\label{eqn:aspa}
 \Sym^2 A_+ \cong \Delta^{2,2}_{prim}\oplus \la \Re(\theta) + 8 \omega^2 \ra.
\end{equation}
Also, the surjection 
$\bigwedge^2_7 \otimes \bigwedge_{21}^2 \too \bigwedge_7^4 \oplus \bigwedge_{35}^4$
 restrics to give a surjection
\begin{equation}\label{eqn:aspa2}
 A_+ \otimes A_- \surj \Delta^{1,1}_{prim}\omega \oplus \la \Im(\theta) \ra.
\end{equation}
More explicitly, take the following elements of $\bigwedge^{2,0}$
 $$
 \left\{ \begin{array}{l} 
c_1  = dz_{12}+dz_{34}\\
c_2 = i\,dz_{12}-i\, dz_{34}\\
c_3 = dz_{13}-dz_{24}\\
c_4 = i\,dz_{13}+i\,dz_{24}\\
c_5 = dz_{14}+dz_{23}\\
c_6 = i\,dz_{14}-i\,dz_{23}
\end{array}\right.
\qquad 
 \left\{ \begin{array}{l} 
c_1' = i\,dz_{12}+i\, dz_{34}\\
c_2' = -dz_{12}+dz_{34}\\
c_3' = i\,dz_{13}-i\,dz_{24}\\
c_4' = -dz_{13}-dz_{24}\\
c_5' = i\,dz_{14}+i\,dz_{23} \\
c_6' = -dz_{14}+dz_{23}
\end{array}\right.
 $$
Then $\gamma_j=\frac12(c_j+\bar c_j)$ are an orthogonal basis for $A_+$, and
$\gamma_j'=\frac12(c_j'+\bar c_j')$ are an orthogonal basis for $A_-$.
We have $c_j\wedge c_j=8 \frac\theta4$,  $c_j\wedge c_j'=8i \, \frac\theta4$, and
$c_i\wedge c_j=0$, $c_i\wedge c_j'=0$, for $i\neq j$. 
The generators of $\Sym^2 A_+$ are $\gamma_i\wedge \gamma_j$,
and the generators of the image of $A_+\otimes A_-$ are
$\gamma_i\wedge \gamma_j'$. From here 
(\ref{eqn:aspa}) and (\ref{eqn:aspa2}) follow easily.
\end{remark}

\begin{remark} \label{rem:Spin6-transitive}
The group $\SU(4)$ acts on $A_\pm \cong \RR^6$, compatibly with 
$\SU(4)\cong \Spin(6)\twoheadrightarrow \SO(6)$. In particular, the
action is transitive on the elements of fixed norm.
\end{remark}

\medskip

Now let $M$ be a compact Calabi-Yau $4$-fold. That is, $M$ is a complex
manifold of (complex) dimension $4$, endowed with a K\"ahler
metric $g$ given by a K\"ahler form $\omega\in \Omega^{1,1}(M)$,
and a holomorphic volume form $\theta \in \Omega^{4,0}(M)$. Both
$\omega$ and $\theta$ are parallel with respect to the Levi-Civita
connection $\nabla$. Therefore the holonomy $\Hol_g \subset
\SU(4)$.

Such $M$ gets an induced $\Spin(7)$-structure, given by the choice
$\Omega= \frac12 \omega^2 + \Re(\theta)$.

Conversely, let $M$ be a manifold with a $\Spin(7)$-structure. To
get an $\SU(4)$-structure, we need to choose a $2$-form $\omega$
which is a section of the sphere bundle $S(\bigwedge^2_7(TM))$.
This gives a complex structure $J$ defined point-wise by Lemma
\ref{lem:1}, making $M$ an almost-complex manifold. Then there is
a $(4,0)$-form $\theta$ defined by $\Re(\theta) =\Omega-\frac12
\omega^2$. If $\omega$ is closed and $J$ is integrable then $M$
is K\"ahler. In this case, $\theta$ is also closed, and we have 
the holonomy contained in $\SU(4)$, i.e., $M$ is a Calabi-Yau $4$-fold.

\section{$\Spin(7)$-connections and holomorphic bundles} \label{sec:connections}

\subsection{$\Spin(7)$-connections}

Let $M$ be an $8$-dimensional $\Spin(7)$-manifold. 
Let $E\to M$ be a complex bundle of rank
$r$, and let $c_i=c_i(E)$ denote its Chern classes. We put a
hermitian metric on $E$, so that $E$ is a $\U(r)$-bundle.

Let $\cA_E$ denote the space of $\U(r)$-connections on $E$ such
that $\tr F_A  \in \Omega^2(M)$ is harmonic, where for
$A\in\cA_E$, we denote the curvature $F_A\in \Omega^2(\End E)$.

Let $\su_E\subset \End(E)$ be the associated bundle of skew-hermitian
traceless endomorphisms, and denote
 $$
 F_A^o:= F_A - \frac1r (\tr F_A) \id \in \Omega^2(\su_E).
 $$
By Chern-Weil theory,
 $$
 \int_M \tr (F_A^o\wedge F_A^o) = 8\pi^2 \left(c_2 - \frac{r-1}{2r}c_1^2 \right)\, .
 $$
We set
 $$
 \beta=\beta(E):= \frac{1}{8\pi^2} [\tr(F_A^o\wedge F_A^o)]= c_2 - \frac{r-1}{2r}c_1^2 \in H^4(M).
 $$

The decomposition $\Omega^2=\Omega^2_7\oplus \Omega^2_{21}$ gives rise to
projections $\pi_7$ and $\pi_{21}$. 

\begin{definition} We say that $A$ is a \emph{$\Spin(7)$-instanton} if it satisfies the
$\Spin(7)$-instanton equation
  \begin{equation}\label{eqn:spin7}
   \pi_7 (F_A^o)=0\, .
  \end{equation}
\end{definition}
  
We have, for $A\in \cA_E$,
$$
 \int_M \tr(F_A^o\wedge F_A^o) \wedge \Omega = 8\pi^2 \beta \cup [\Omega] \, .
 $$
Also
  $$
 \int_M \tr(F_A^o\wedge F_A^o) \wedge \Omega = ||\pi_{21} F_A^o||^2- 3||\pi_7F_A^o||^2\, ,
 $$
Here we have used that $\tr( F_A^o \wedge F_A^o) = - \tr( (\overline{F^o_A})^t
\wedge F_A^o)$, taking the 
transpose conjugate on matrices, and using $\alpha\wedge\alpha\wedge
\Omega= 3|\alpha|^2 \vol$ on $\bigwedge^2_7$,
and $\alpha\wedge\alpha\wedge \Omega= -|\alpha|^2 \vol$ on $\bigwedge^2_{21}$.

We have the Yang-Mills functional 
 $$
  \cF(A)= ||F_A^o||^2=||\pi_{21} (F_A^o)||^2 + ||\pi_7 (F_A^o) ||^2\, .
 $$
The minimum of the functional $\cF$ is attained for
$\Spin(7)$-instantons. Such minimum is the topological invariant
$[8\pi^2(c_2- \frac{r-1}{2r} c_1^2)] \cup [\Omega]$. The Yang-Mills functional is
gauge-invariant, under the gauge group $\cG=\Aut(E)$. In  \cite{Reyes}, 
Reyes-Carri\'on studies the linearization of the $\Spin(7)$-equation,
which is the elliptic complex
 $$
 \Omega^0(\su_E)\to \Omega^1(\su_E)\to \Omega^2_7(\su_E)\, .
 $$
There should be a moduli space of $\Spin(7)$-instantons, but it
has not been constructed yet. In \cite{Tian} it is discussed how
the $\Spin(7)$-connections may blow-up for a sequence of $\Spin(7)$-instantons,
which is needed to compactify the
moduli space. However, the regularity of the moduli spaces
(transversality of the $\Spin(7)$-equation) is a difficult issue,
since the possible perturbations of metrics keeping the holonomy
in $\Spin(7)$ are too scarce.

In this paper, we shall use the equation (\ref{eqn:spin7}) as it
is, without changing the
$\Spin(7)$-structure but changing the underlying
$\SU(4)$-structure. This does not affect the property of (\ref{eqn:spin7}) having
solutions. We will call this procedure
\emph{$\Spin$-rotation}.

\subsection{Holomorphic structures}

Let $(M,\omega)$ be a K\"ahler manifold, that is, a complex manifold
endowed with a $\U(4)$-structure. Suppose that $E\to M$ is a complex
rank $r$ bundle with a hermitian metric, and let $A$ be a $\U(r)$-connection.
We say that the connection $A$ is Hermitian-Yang-Mills (w.r.t.\ $\omega$) if
  $$
  \left\{ \begin{array}{l} 
           F_A\in \Omega^{1,1}(\End E)\, , \\
\frac{i}{2\pi}\la F_A ,\omega \ra =\lambda \, \id\, ,
\end{array} \right.
  $$
where $\lambda$ is a constant. This constant is determined by
  $$
  \frac{i}{2\pi} \int \tr F_A \wedge \omega^3= [c_1(E) ]\cup[\omega]^3 = \lambda
  [\omega]^4\, .
  $$
For a $\omega$-HYM connection, we have 
$F_A^{0,2}=\bar\partial_A^2=0$, so $\bar\partial_A$ defines
a holomorphic structure on $E$.

\begin{proposition}\label{prop:HYM=stable}
  Let $(M,\omega)$ be a K\"ahler manifold.
  The following are equivalent:
  \begin{itemize}
  \item $E$ admits a Hermitian-Yang-Mills connection w.r.t.\ $\omega$.
  \item $E$ admits a holomorphic structure $\bar\partial_A$ and
  $(E,\bar\partial_A)$ is polystable w.r.t.\ $\omega$.
  \end{itemize} \hfill $\Box$
\end{proposition}

Here let $(M,\omega)$ be a K\"ahler manifold, denote
$H=[\omega]\in H^2(M)$ the class of the K\"ahler form. We
define
  $$
  \deg_H (E)=\la c_1(E) \cup  [\omega]^3, [M] \ra \, .
  $$
We define the slope of a bundle $E$ as
  $$
  \mu(E) := \frac{\deg_H (E)}{\rk(E)}\, .
  $$
A holomorphic bundle $\cE=(E,\bar\partial_A)$ is stable if
$\mu(\cF)<\mu(\cE)$ for any proper holomorphic subsheaf
$\cF\subset \cE$. It is polystable if it is the direct sum of
stable bundles of the same slope.

\begin{definition}
We say that a connection $A$ is traceless Hermitian-Yang-Mills (w.r.t.\ $\omega$) if
$\tr F_A$ is harmonic and
  $$
  F_A^o\in \Omega^{1,1}_{prim}(\su_E)\, .
  $$
We say that the bundle $E\to M$ is traceless $\omega$-HYM if it admits a 
traceless $\omega$-HYM connection.
\end{definition}

\begin{remark}
 We can prove that for a K\"ahler manifold $(M,\omega)$, and $E\to M$ a $\U(r)$-bundle with
 a connection $A$, $A$ is $\omega$-HYM $\iff$
 $A$ is traceless $\omega$-HYM and $c_1(E)  \in H^{1,1}(M)$. However, we shall not need this.
\end{remark}

Note that for a traceless $\omega$-HYM connection, the first Chern class $c_1(E)\in H^2(M,\ZZ)$
may be arbitrary. There are not restrictions with respect to the Hodge decomposition.

\begin{lemma} \label{lem:9}
Let $(M,\omega)$ be a projective complex manifold, and $A$ a traceless $\omega$-HYM
connection on a bundle $E$ with $\beta=\beta(E)$. 
Then $\beta$ is an algebraic class (i.e.\ defined by a rational algebraic cycle).
\end{lemma}

\begin{proof}
Consider the bundle $\su_E\subset \End E$. This has a product connection
$\tilde A=A\otimes  \id - \id \otimes A^t$. This connection has zero trace, and it is
traceless $\omega$-HYM, since  
 $$
  F_{\tilde A}(\phi)=[F_A, \phi]\in \Omega^{1,1}_{prim}(\su_E),
  $$
for $\phi \in \Gamma(\su_E)$. So $F_{\tilde A}\in \Omega^{1,1}_{prim}(\su_E)$ (the
adjoint bundle of $\su_E$ is $\su_E$ again).
Hence $\tilde A$ is $\omega$-HYM. Therefore $\su_E$ is a holomorphic vector bundle, which is
moreover polystable. Also $c_1(\su_E)=0$ and $c_2(\su_E)= 2r c_2 - (r-1) c_1^2=2r\,\beta$.
So $\beta(\su_E)=2r\, \beta$. Hence $\beta$ is an algebraic class.
\end{proof}

\begin{proposition}[\cite{Mistretta}] \label{prop:Mistretta}
 Let $(M,\omega)$ be a projective complex manifold, $H=[\omega]$.
 Let $\beta_0\in H^{2,2}(M)$ be an algebraic class (a Hodge class
 representable by an algebraic cycle). Then there exists $N\in \ZZ$
 and a $\omega$-HYM bundle $E\to M$ with $\beta(E)=\beta_0 + N H^2$. \hfill $\Box$
\end{proposition}

\subsection{Relationship of $\omega$-HYM and $\Spin(7)$-connections}

Now let $(M,\omega)$ be a Calabi-Yau manifold, i.e.\ $M$ has an
$\SU(4)$-structure, and therefore an induced $\Spin(7)$-structure.
Suppose that $E\to M$ is a complex rank $r$ bundle with a
hermitian metric, and let $A$ be a $\U(r)$-connection.
The following result appears already in \cite{Lewis} with 
a slightly different formulation.

\begin{proposition} \label{prop:HYM=Spin7}
 We have the following equivalent statements:
 \begin{itemize}
 \item[(a)] $A$ is a traceless $\omega$-Hermitian-Yang-Mills connection.
 \item[(b)] $A$ is a $\Spin(7)$-connection and $\beta\in H^{2,2}(M)$.
 \item[(c)] $A$ is a $\Spin(7)$-connection and $\beta \cup [\Re(\theta)]=0$.
 \end{itemize}
\end{proposition}

\begin{proof}
First assume (a). Denote $R=F_A^o$. As $A$ is traceless $\omega$-HYM, we have
that $R^{2,0}=0$ and $\la R,\omega \ra=0$. Now the decomposition $\bigwedge^2_7=\la \omega
  \ra\oplus A_+$ and $A_+\subset \bigtriangleup^{2,0}$,
  give us that $\pi_7(R)=0$, as required. Moreover,
  $\tr( R\wedge R)$ represents $\beta$, up to a factor, and
  this is of type $(2,2)$. (b) follows.

The implication (b) $\Rightarrow$ (c) is clear.

Finally, assume (c). So $R=F_A^o \in\bigwedge^2_{21}$. This means that
  $\la R,\omega\ra=0$ and $\pi_{A_+}(R^{2,0}+R^{0,2})=0$. So
  $L(R^{2,0})=-R^{0,2}=-\overline{R^{2,0}}$. This implies that
  $R^{2,0}\wedge R^{2,0}=- \frac14 |R^{2,0}|^2\, \theta$ and hence
  $$
  \begin{aligned}
  0 = 8\pi^2 \beta \cup [\Re(\theta)] = & \int_M \tr(R\wedge R)\wedge \Re(\theta) \\
  =& \int_M \tr (R^{2,0} \wedge R^{2,0} + R^{0,2}\wedge R^{0,2})\wedge (\theta+\bar{\theta})/2 \\
  =& - 2 \int_M  \frac14 |R^{2,0}|^2 \, \frac{\theta\wedge\bar{\theta}}{2} \\
  =&  - 4 \, ||R^{2,0}||^2 \, .
  \end{aligned}
  $$
   So $R^{2,0}=0$ and $A$ is traceless $\omega$-Hermitian-Yang-Mills.
\end{proof}

\section{Rotation of complex structures on the $8$-torus} \label{sec:8torus}

Now we shall study explicitly the case of the $8$-torus.
Let $V=\RR^8$ be an $8$-dimensional real vector space, and let
$\Lambda \subset V$ be a lattice. The group $\GL(8, \RR)$
acts on $V$. 
We consider the torus 
 $$
 X=V/\Lambda,
 $$ 
and let $g$ be a Riemannian flat metric on $X$. This is the same as a
metric on $V$, i.e.\ a reduction $\SO(8)\subset \GL(8,\RR)$. Note that the set
of (flat) metrics $g$ is parametrised by $\GL(8,\RR)/\SO(8)$.

A complex structure $J$ on $X$ which makes $X$ into a complex torus
is a complex structure $J:V\to V$, $J^2=-\id$. This is a reduction to
$\GL(4,\CC)\subset \GL(8,\RR)$. If we require that $J$
be compatible with a metric $g$, then we are asking for a
reduction to $\U(4)\subset \SO(8)$.
Alternatively, if there is a complex structure $J$, and we want to
put a K\"ahler metric on $X$, we look for subgroups $G\cong
\SO(8)$ containing $\U(4)$. This is parametrized by an open cone
in $\Re (\bigwedge^{1,1}) \cong \RR^{16}$. This set is $\GL(4,\CC)/\U(4)$.

If we have a complex structure $J$ with a K\"ahler metric $\omega$, and we
want to get an $\SU(4)$-structure, then we have to fix a generator
$\theta\in \bigwedge^{4,0}$ with $|\theta|=4$. This is
parametrized by $\U(4)/\SU(4)=S(\bigwedge^{4,0})$.

Now suppose that we have an $\SU(4)$-structure
$(J,\omega,g,\theta)$. If we conjugate by an element $\phi\in
\Spin(7)$, then we get a new $\SU(4)$-structure
$(J',\omega',g,\theta')=\phi_*(J,\omega,g,\theta)$, where
 $$
 \Omega=\frac12 \omega^2+\Re(\theta)=\frac12 (\omega')^2+\Re(\theta').
 $$
By Lemma \ref{lem:1}, the possible K\"ahler forms are elements of 
$\bigwedge^2_7=\la \omega\ra \oplus
A_+$ of norm $2$. Therefore
 \begin{equation}\label{eqn:omega'}
 \omega'=2\, \frac{\omega+\gamma}{|\omega+\gamma|}
 \end{equation}
and $\gamma\in A_+$ (note that the metric is fixed in this process, so
$| \cdot |$ has a clear meaning). 
Note that $\gamma=\frac12 (c +\bar c )$, $c \in \bigwedge^{2,0}_{\omega}$. The condition
$\gamma\in A_+$ is rewritten as $L(c )=\bar c $, which in turn can be rewritten as
$c  \wedge c  = |c |^2 \, \frac\theta4$. There is a possible $\theta$ satisfying this
if and only if $|c  \wedge c | = |c |^2$.

On a torus $X=V/\Lambda$ there is an isomorphism $H^{i,j}(X)\cong \bigwedge^{i,j} V=
 \bigwedge^{i,j}$, which identifies the cup product and the wedge on $\bigwedge^{*} V$.
We shall use this isomorphism implicitly.

\begin{proposition} \label{prop:rotation}
  Let $X$ be an $8$-torus with an $\SU(4)$-structure. Let $E\to X$ be a $\U(r)$-bundle
with $\beta=\beta(E)\in H^4(X)$ and $A\in \cA_E$.
Take $\gamma\in A_+$, $\gamma=\frac12 (c +\bar c )$, $c \in \bigwedge^{2,0}_{\omega}$,
and $\omega'= 2 (\omega+\gamma)/|\omega+\gamma|$.

Suppose $\beta \in \bigwedge^{2,2}_{\omega}$. Then $A$ is traceless $\omega'$-HYM if and only
if the following conditions are satisfied:
\begin{itemize}
 \item $A$ is $\Spin(7)$-instanton,
 \item $\beta \wedge c  \wedge \bar c = \frac14 |c |^2 (\beta \wedge\omega^2)$,
 where $\frac12 \omega^2 \wedge c  \wedge \bar c = |c |^2 \vol$.
\end{itemize}
\end{proposition}

\begin{proof}
By Proposition \ref{prop:HYM=Spin7}, $A$ is traceless $\omega'$-HYM
$\iff$ $A$ is a $\Spin(7)$-instanton
and $\beta \cup [\Re(\theta')]=0$. So we need to see that this latter
condition is satisfied. 

As $\beta$ is of type $(2,2)$ w.r.t.\ $\omega$, we have that 
$\beta \cup [\Re(\theta)]=0$. The equality 
$\beta \wedge \Re(\theta)=\beta \wedge \Re(\theta')=0$ is equivalent to
 $$
 \beta \wedge \omega^2= \beta \wedge (\omega')^2\, ,
 $$
since $\Omega= \frac12 \omega^2 + \Re(\theta)=\frac12 (\omega')^2+ \Re(\theta')$ is fixed.
We rewrite this as
 $$
 \frac14 |\omega+\gamma|^2 (\beta \wedge \omega^2)= \beta \wedge (\omega+\gamma)^2\, .
 $$
Now $|\omega+\gamma|^2= |\omega|^2+|\gamma|^2 = 4+ |\gamma|^2$, since $\bigwedge^2_7=\la \omega\ra \oplus A_+$
is orthogonal. So we get
 $$
 \frac14 |\gamma|^2 (\beta \wedge \omega^2)= 2\beta \wedge \omega \wedge \gamma + \beta \wedge\gamma^2 =
\beta \wedge\gamma^2 \, ,
 $$
since $\beta \wedge \omega \wedge \gamma=0$, because $\beta$ is of type $(2,2)$ (w.r.t.\ $\omega$).

Now using that $\gamma=\frac12 (c + \bar{c})$, we rewrite this as
 $$
 \frac14 |c |^2 (\beta \wedge \omega^2)= \beta \wedge c  \wedge \bar c  \, ,
 $$
as required.
\end{proof}

\begin{definition} \label{def:new}
  A cohomology class $\beta\in H^4(X,\ZZ)$ is \emph{traceless $\omega$-HYM} if
 there is a bundle $E\to X$ with $\beta=\beta(E)$,
 and a connection $A\in \cA_E$ which is traceless $\omega$-HYM.
\end{definition}

\begin{theorem}\label{thm:rotation}
 Let $X$ be an $8$-torus with an $\SU(4)$-structure.
 Take $\gamma\in A_+$, $\gamma=\frac12 (c +\bar c )$,
 and $\omega'=2(\omega+\gamma)/|\omega+\gamma|$. Let $\beta \in H^4(X,\ZZ)$. If
 \begin{itemize}
  \item $\beta$ is traceless $\omega$-HYM.
  \item $\beta \wedge c  \wedge \bar c  = 6k \, |c |^2 \vol$, where $\beta\wedge \omega^2= 24k\, \vol$
  and $\omega^2 \wedge c  \wedge \bar c = 2|c |^2\vol$.
 \end{itemize}
Then $\beta$ is traceless $\omega'$-HYM.
 \end{theorem}

\begin{proof}
 Take a bundle $E\to X$ with $\beta=\beta(E)$
 and a $\omega$-HYM connection $A$ on $E$. 
 Then $A$ is a $\Spin(7)$-connection and of type $(2,2)$ by
 Proposition \ref{prop:HYM=Spin7}. Now  $\beta \wedge c  \wedge \bar c  = 6k \, |c |^2 \vol=
 \frac14 |c |^2 (\beta\wedge \omega^2)$. Proposition \ref{prop:rotation} gives the result. 
 \end{proof}

The equality defined in the second line of Theorem \ref{thm:rotation} is rewritten as
 \begin{equation}\label{eqn:main}
  ( \beta -3k\, \omega^2) \wedge\gamma^2=0\, ,
 \end{equation}
where 
 $$
 \beta = \beta_0 + \omega\wedge \beta_1 + k\, \omega^2,
 $$ 
$\beta_0,\beta_1$ are primitive $(2,2)$
and $(1,1)$-forms, respectively.

\begin{remark} \label{rem:rotate-back}
Equation (\ref{eqn:main}) has a natural symmetry
when we \emph{rotate back} from
$\omega'$ to $\omega$. First, we compute the value $k'$ of $\beta$ with respect to
$\omega'$.
 $$
 \alpha \wedge (\omega')^2 = \alpha \wedge \left( 2 \frac{\omega+\gamma}{|\omega+\gamma|}\right)^2 
= 4\alpha \wedge \frac{\omega^2+2 \gamma\wedge \omega+\gamma^2}{4+|\gamma|^2} 
= 4  \frac{24 k  + 6k |\gamma|^2}{4+|\gamma|^2} \vol = 24k\, \vol ,
 $$ 
and 
 $$ 
(\omega')^4 = \frac{16(\omega+\gamma)^4}{|\omega+\gamma|^4}  
=\frac{16(24 + 6 \cdot 2 |\gamma|^2 + \frac32 |\gamma|^4)}{(4+|\gamma|^2)^2}\vol  
= \frac{24\cdot 16 ( 1+ \frac14|\gamma|^2)^2 }{(4+|\gamma|^2)^2}  \vol =24 \vol,
 $$
where we have used that
 $$
 \gamma^4 = \frac{1}{16} (c +\bar c )^4 = \frac{6}{16} c ^2\wedge \bar c ^2 
 = \frac{6}{16} |c |^4 \frac{1}{16} \theta \wedge \bar{\theta} 
 =\frac32 |\gamma|^4 \vol.
 $$
So $k'=k$. Now to rotate back, we have to take
 $$
  \gamma'= - \frac{2}{|\omega+\gamma|} \left( \gamma- \frac14 |\gamma|^2\omega\right) \in A_+'\, ,
  $$
where $\bigwedge^2_7=\la \omega'\ra \oplus A_+'$ orthogonally
(note that $|\gamma'|=|\gamma|$ and $\la \omega',\gamma'\ra=0$). Then 
 $$
 \omega= 2\frac{\omega'+\gamma'}{|\omega'+\gamma'|}\, .
 $$
We compute
 \begin{align*}
  ( \alpha - & 3k\, (\omega')^2) \wedge (\gamma')^2 =
  \left( \alpha -3k\, \left(2 \frac{\omega+\gamma}{|\omega+\gamma|}\right)^2 \right)
  \wedge \frac{4}{|\omega+\gamma|^2} \left( \gamma- \frac14 |\gamma|^2\omega\right)^2 \\
  &=
 \frac{4}{|\omega+\gamma|^2} \left(
 \alpha \wedge \left( \gamma- \frac14 |\gamma|^2\omega\right)^2
 -12 k\, \left(\frac{\omega+\gamma}{|\omega+\gamma|}\right)^2
  \wedge \left( \gamma- \frac14 |\gamma|^2\omega\right)^2 \right)\\
  &=
 \frac{4}{|\omega+\gamma|^2} \left(
 \left(6k |\gamma|^2 + \frac{1}{16}|\gamma|^4 24k \right)
 -  \frac{12k}{|\omega+\gamma|^2}
\left(\frac{1}{16} |\gamma|^4 \omega^4 + \gamma^4 +(1-|\gamma|^2+\frac{1}{16}|\gamma|^4) \omega^2\wedge\gamma^2\right)
\right)\\
  &= \frac{4}{4+|\gamma|^2} \left( \frac64 k |\gamma|^2 (4  + |\gamma|^2 )
 - \frac{12k}{4+|\gamma|^2}\left(\frac{24}{16} |\gamma|^4  +
 \frac32 |\gamma|^4 + 2 |\gamma|^2(1-|\gamma|^2+\frac{1}{16}|\gamma|^4) \right) \right)\\
  &= \frac{4}{4+|\gamma|^2} \left( \frac64 k |\gamma|^2 (4  + |\gamma|^2 )
 -\frac{24k}{4+|\gamma|^2}  (1 + \frac14 |\gamma|^2)^2  |\gamma|^2\right)  
 =0\, ,
 \end{align*}
so equation (\ref{eqn:main}) is satisfied w.r.t.\ $\omega'$, as expected.
\end{remark}

\medskip

We shall need the following later.

\begin{definition}
We say that $\beta\in H^4(X,\QQ)$ is asymptotically traceless $\omega$-HYM if there
is some $N\gg 0$ such that $N\, \beta$ is traceless $\omega$-HYM.
\end{definition}

A class $\beta$ that is asymptotically traceless $\omega$-HYM in a projective
manifold $(X,\omega)$ is a (rational) algebraic cycle. By Definition \ref{def:new}, there
is some bundle $E\to X$ such that $N\, \beta =\beta(E)$, $N\gg 0$, and a connection 
$A\in\cA_E$ which is $\omega$-HYM. Therefore $N\, \beta$ is an algebraic cycle
by Lemma \ref{lem:9}, and thus $\beta$ is a rational algebraic cycle.

\smallskip

Let $X=\CC^4/\Lambda$, $X'=\CC^4/\Lambda'$ be isogenous complex tori,
then the asymptotically traceless HYM cycles are the same. We may suppose
that $\pi:X'\to X$ is a finite covering, corresponding to an inclusion
$\Lambda' \subset \Lambda \subset V=\CC^4$. There is a natural isomorphism
$\pi^*:H^*(X,\QQ) \stackrel{\cong}{\too} H^*(X',\QQ)$.

\begin{lemma}\label{lem:asymp}
Fix K\"ahler forms $\omega$ on $X$ and $\omega'=\pi^*\omega$ on $X'$.
Under $\pi^*$, the asymptotically traceless HYM classes on $X$ and on $X'$ are the same.
\end{lemma}

\begin{proof}
If $\beta\in H^4(X)$ is traceless $\omega$-HYM, then $\pi^*\beta \in H^4(X')$ is $\omega'$-HYM. 
Just pull-back the bundle and connection under $\pi:X'\to X$. 

If $\beta' \in H^4(X')$ is $\omega'$-HYM. Let $E'$ be the corresponding
holomorphic polystable bundle with $\beta'=\beta(E)$. 
Assume for the moment that $E'$ is stable.
Consider $E=(\pi_* E')^G$, where 
$G=\Lambda/\Lambda'$, which is a finite group. Then $E$ is a holomorphic bundle
such that $\pi^* E\cong E'$. $E$ is stable with respect to $\omega$: consider
a subsheaf $F\subset E$, then $\pi^*F\subset E'$, hence $\deg_{\omega'} \pi^*F < 0$
and thus $\deg_\omega F < 0$. Finally, if $E'$ is a direct sum of stable bundles, then
the same holds for $E$.
\end{proof}

\section{A Bogomolov type inequality for complex $4$-tori}

Let us consider a K\"ahler $4$-torus $X=\CC^4/\Lambda$ 
with K\"ahler form $\omega \in \bigwedge^{1,1}$. Let $E\to X$ be 
an $\omega$-polystable holomorphic bundle. The celebrated Bogomolov
inequality says that \cite{Bando-Siu} 
  \begin{equation}\label{eqn:bogo}
  \left( c_2-\frac{r-1}{2r}c_1^2 \right) \cup [\omega]^2 \geq 0.
  \end{equation}
We want to strengthen this inequality by using the theory of
$\Spin(7)$-instantons developed in previous sections.

First of all, note that the (poly)stability condition is not
affected by tensoring $E$ with a line bundle $L$. The invariant
 $$
\beta(E)=c_2(E)-\frac{r-1}{2r}c_1(E)^2 \in H^4(M,\QQ)
 $$
satisfies that $\beta(E\otimes L)=\beta(E)$ for any line bundle
$L$ (actually it is the only such invariant of degree $4$ up
to a scalar multiple).

As $E$ is a $\omega$-polystable bundle, it admits a $\omega$-HYM
connection $A$ by Proposition \ref{prop:HYM=stable}. Note that $\beta \in \bigwedge^{2,2}$. We decompose 
 \begin{equation}\label{eqn:k}
 \beta= \beta_0 + \beta_1\wedge\omega + k\, \omega^2\, ,
 \end{equation}
where $\beta_0\in \bigwedge\nolimits^{2,2}_{prim}$
and $\beta_1\in \bigwedge\nolimits^{1,1}_{prim}$. 
Bogomolov inequality (\ref{eqn:bogo}) says that 
 $$
 k\geq 0.
 $$
Moreover, if $k=0$  then $\beta=0$ (see \cite{Bando-Siu}).

\begin{proposition} \label{prop:17}
 Fix an $\SU(4)$-structure on $X$, by fixing $\theta \in \bigwedge^{4,0}$. 
 Let $\beta\in\bigwedge^{2,2}$ be an $\omega$-HYM class.
 Then we have 
 \begin{itemize}
 \item either $(\beta-3k\, \omega^2)\wedge \gamma^2\leq 0$, for all $\gamma\in A_+$, 
 \item or $(\beta-3k\, \omega^2)\wedge \gamma^2 \geq 0$, for all $\gamma\in A_+$. 
 \end{itemize}
 ($k$ is the value in (\ref{eqn:k}) and $A_+\subset \bigtriangleup^{2,0}$ is defined in
 (\ref{eqn:A+})).
\end{proposition}

\begin{proof}
 Take $E\to X$ a bundle with $\beta=\beta(E)$ and $A\in \cA_E$ a $\omega$-HYM connection.
Then $A$ is a $\Spin(7)$-instanton. Consider the $\Spin(7)$-rotation equation (this is
Equation (\ref{eqn:main}))
 \begin{equation}\label{eqn:constraint}
 (\beta- 3k\, \omega^2) \wedge \gamma\wedge\gamma =0,
 \end{equation}
for $\gamma\in A_+$, 
and let $\cQ_\omega \subset A_+$ be the space of solutions to (\ref{eqn:constraint}).
This is a (real) quadric in $A_+ \cong \RR^6$, where the point $\gamma=0$
(corresponding to $\omega$) is a singular point. 
The set of solutions to (\ref{eqn:constraint}) is reduced to the origin 
if and only if the quadratic form is definite (positive or negative). 
Let $\widehat\cQ_\omega$ be the closure of 
 $$
 \left\{ 2 \frac{\omega+\gamma}{|\omega+\gamma|} \,  | \, \gamma \in \cQ_\omega \right\} 
  \subset S(\bigwedge\nolimits_7^2).
 $$

Suppose that $\gamma \neq 0$ is a solution to (\ref{eqn:constraint}). Then 
$\omega'=2(\omega+\gamma)/|\omega+\gamma|$ 
gives another complex structure for $X$ for which $\beta$ is traceless $\omega'$-HYM,
by Proposition \ref{prop:rotation}. The set
of solutions to the $\Spin(7)$-rotation equation (cf.\ Remark \ref{rem:rotate-back})
 $$ 
 (\beta- 3k'\, (\omega')^2) \wedge \gamma\wedge\gamma =0,
 $$
for $\gamma\in A_+'$ should give the same set of complex structures as (\ref{eqn:constraint}), 
that is $\widehat\cQ_\omega=\widehat\cQ_{\omega'}$. 
In particular $\omega'$ should be a singular point of this set. The conclusion is that
all points in $\cQ_\omega$ are singular, i.e., that $\cQ_\omega$ is a 
linear subspace. Equivalently, the defining equation  
(\ref{eqn:constraint}) is semi-definite. 
\end{proof}

\medskip 

It is of interest to determine the set of (rational) classes 
which are represented as $\beta(E)$ for
polystable bundles $E$. We define
 $$
 \cH_\omega = \{\beta \in H^{2,2}(X)\, | \, \beta \text{ is asymptotically $\omega$-HYM}\}.
 $$

By Lemma \ref{lem:asymp}, if $\pi: X\to X'$ is an isogeny between $4$-tori, and $\omega=\pi^*\omega'$,
then $\cH_{\omega}=\pi^* \cH_{\omega'}$.

\begin{lemma}
 For any K\"ahler torus $X$, the set $\cH_{\omega}$ is a convex cone.
\end{lemma}

\begin{proof}
It is enough to see that if $\beta_1,\beta_2$ are $\omega$-HYM, then $\beta_1+\beta_2$ is $\omega$-HYM.
Just take bundles $E_1,E_2\to X$ which are polystable, and with trivial determinant,
with $\beta(E_i)=\beta_i$, $i=1,2$. Then $E_1\oplus E_2$ is polystable and $\beta(E)=\beta_1+\beta_2$.
\end{proof}

Suppose from now on
that $(X,\omega)$ is an abelian variety.

\begin{proposition} \label{prop:bogomolov}
 If $X$ is an abelian variety and $\beta \in  \cH_\omega$ then
 $$
 (\beta-3k\, \omega^2)\wedge \gamma\wedge\gamma\leq 0,
 $$ 
 for all $\gamma\in A_+$, where $k$ is given by (\ref{eqn:k}). 
 If there is equality, then $\beta$ is 
 $\Spin$-rotable via $\gamma$.
\end{proposition}

\begin{proof}
We have to see that (\ref{eqn:constraint}) is \emph{negative} semi-definite.
For $\beta\in \cH_\omega$, consider the quadratic form on $A_+$,
 $$
  \Phi_\beta(\gamma,\gamma)=(\beta- 3k_\beta \, \omega^2) \wedge \gamma\wedge\gamma\, ,
 $$
where $k_\beta$ is given, as usual, by $\beta\wedge \omega^2=k_\beta\, \omega^4$.

Let $\beta,\beta'\in \cH_\omega$. 
Then $t\beta+(1-t)\beta'\in \cH_\omega$ for $t\in[0,1]\cap \QQ$.
The form 
 $$
 \Phi_{t\beta+(1-t)\beta'}=t\Phi_\beta+(1-t)\Phi_{\beta'}
 $$
is semidefinite for all $t\in [0,1]\cap \QQ$. This implies that either $\Phi_\beta$ and 
$\Phi_{\beta'}$ are semidefinite of the same type (positive or negative) or
that there is some $\lambda\in \RR$ with $\Phi_\beta=\lambda \Phi_{\beta'}$.
By (\ref{eqn:aspa}), the set $\gamma\wedge\gamma$ expands $\bigwedge^{2,2}_{prim}$.
So in the second case it must be $\beta=\lambda \beta'$. 
Therefore $\lambda$ must be rational and positive,
and $\Phi_\beta$, $\Phi_{\beta'}$ are semidefinite of the same type.

Now we use that $X$ is an abelian variety. As $X$ is a projective complex manifold,
we apply Proposition \ref{prop:Mistretta} to $\beta_0=0$. Then there exists $k\gg 0$ such
that there is an $\omega$-HYM bundle $E\to X$ with $\beta(E)=N\, \omega^2$. This
means that $\omega^2$ is asymptotically $\omega$-HYM. Clearly 
$\Phi_{\omega^2}=-2\omega^2 <0$, therefore all $\Phi_\beta$ are negative semidefinite.
\end{proof}

\begin{remark} \label{ref:conjecture-semi-def-neg}
Proposition \ref{prop:bogomolov} also holds if 
$X$ is a complex torus with $H_{prim}^{1,1}(X)\cap H^2(X,\QQ)\neq 0$. 
Take a non-zero $\alpha\in H^{1,1}_{prim}(X)\cap H^2(X,\QQ)$, and let 
$\cL$ be a line bundle with $c_1(\cL)=\alpha$. Then $E=\cL \oplus \cL^{-1}$
is a polystable rank $2$ bundle with $\beta=-\alpha^2$. By Lemma \ref{lem:alpha2},
$\Phi_{-\alpha^2}$ is negative semidefinite.
\end{remark}

We conjecture that the statement in Proposition \ref{prop:bogomolov} also holds 
for any K\"ahler (non-projective) complex $4$-tori. That is, if $(X,\omega)$ is
a K\"ahler $4$-torus, and $\beta\in \cH_\omega$, then $(\beta-3 k\, \omega^2)\wedge \gamma^2\leq 0$,
for any $\gamma\in A_+$, and $k=k_\beta$. Equivalently, in 
Proposition \ref{prop:17} it is the first line that holds.

\begin{theorem} \label{thm:bogomolov}
Let $X$ be an abelian variety. 
Consider $\beta_0\in \bigwedge^{2,2}_{prim}$. 
Let 
 $$
 k_m=k_m(\beta_0)=\frac14 \min\{ \frac{\beta_0\wedge c \wedge \bar{c}}{\vol} 
 \, | \, c \in \bigwedge\nolimits^{2,0}, \, |c |=1\},
$$
and 
 $$
 \cR =\{ c \in \bigwedge\nolimits^{2,0}\, | \, 
\frac{\beta_0\wedge c \wedge \bar{c}}{\vol} = k_m, \, |c \wedge c |=1, \, |c |=1\}.
 $$

 Then for any $\beta=\beta_0+ \beta_1\wedge \omega+ k\, \omega^2$ which is
 a traceless $\omega$-HYM class, we have $k\geq k_m$. And $k=k_m$ if and only if
$\beta$ is $\Spin$-rotable (via any $\gamma=\frac\lambda2 (c +\bar{c})$, $c\in \cR$, $\lambda\in \RR_{>0}$).
\end{theorem}

\begin{proof}
Let $\beta=\beta_0+ \beta_1\wedge \omega+ k\, \omega^2$ be 
an $\omega$-HYM class. By (\ref{eqn:aspa}), 
$\beta_1\wedge \omega \wedge \gamma\wedge\gamma=0$, for any $\gamma\in A_+$. 
By Proposition \ref{prop:bogomolov}, $(\beta-3k\, \omega^2)\wedge \gamma\wedge\gamma\leq 0$.
Therefore $(\beta_0- 2k \, \omega^2)\wedge \gamma\wedge\gamma\leq 0$, for $\gamma\in A_+$. 
Write $\gamma=\frac12(c +\bar{c})$. Then $\gamma'=\frac12 (ic -i\bar{c})\in A_-$ 
and $(\beta_0- 2k \, \omega^2)\wedge (\gamma')^2 =
(\beta_0- 2k \, \omega^2)\wedge\gamma^2\leq 0$, so the inequality holds on $A_-$.
Now take a general $\gamma\in \triangle^{2,0}$, $\gamma=\gamma_+ + \gamma_-$, with
$\gamma_\pm \in A_\pm$. Then $\beta_0\wedge \gamma_+\wedge\gamma_- =0$ by
(\ref{eqn:aspa2}) and $\omega^2\wedge \gamma_+\wedge\gamma_-= 2 \la \gamma_+,\gamma_-\ra=0$. So  
 $$
 (\beta_0- 2k \, \omega^2)\wedge \gamma^2=(\beta_0- 2k \, \omega^2)\wedge\gamma_+^2
 +(\beta_0- 2k \, \omega^2)\wedge \gamma_-^2\leq 0. 
 $$

Take $\gamma \in \triangle^{2,0}$ and write $\gamma =\frac12(c +\bar{c})$,
$c \in \bigwedge^{2,0}$. Then 
$(\beta_0- 2k \, \omega^2)\wedge c \wedge \bar{c}\leq 0$,
or equivalently
 $$
 4k_m\, \vol \leq \beta_0\wedge c \wedge \bar{c} \leq 
 2k \, \omega^2\wedge c \wedge \bar{c} = 4k\, |c |^2\, \vol= 4 k\, \vol.
 $$ 

If $k=k_m$, then $\beta$ is $\Spin$-rotable for any $\gamma\in A_+$ satisfying equality. But
$\gamma= \frac12 (c +\bar{c})$ is in $A_+$ when $c \wedge c = |c |^2 \frac\theta4$.
This happens for a suitable $\theta$ if and only if $|c \wedge c |= |c |^2$.
\end{proof}

Let $\cD^2\subset H^{2,2}(X)$ be the set of Hodge $(2,2)$-classes represented by algebraic cycles
(all Hodge $(2,2)$-classes if we assume the Hodge Conjecture). Decompose orthogonally
$\cD^2=\cD^2_o \oplus \QQ \omega^2$. Then there is a convex radial function $f:\cD^2_o \to [0,\infty)$ 
(i.e., $f(t \beta)=t \cdot f(\beta)$ for $t>0$)
such that the closure of  $\cH_\omega$ is
 $$
 \text{cl}(\cH_\omega) = \{ \beta + k \,\omega^2 \, |\, \beta \in \cD^2_o, \, k\geq f(\omega)\}.
 $$
This is a consequence of Proposition \ref{prop:Mistretta}:
for any $\beta \in \cD^2_0$, there is
some large $k\gg 0$ for which $\beta+k\,\omega^2$ is $\omega$-HYM. 

The usual Bogomolov inequality says that $f\geq 0$. Theorem \ref{thm:bogomolov}
says that 
 $$
f(\beta_0+\beta_1\wedge \omega)\geq \max\{k_m(\beta_0),0\}.
 $$
Note that $k_m(\beta_0)$ may be negative, in which case there is no improvement.
However, if $k_m(\beta_0)<0$ then $k_m(-\beta_0)>0$, so $f\neq 0$.

In \cite{Voisin}, C.\ Voisin finds K\"ahler non-projective $4$-tori for which the Hodge 
Conjecture fails in the following extended version: $X$ has a Hodge $(2,2)$-class not
represented by Chern classes of sheaves. Such examples do not have stable bundles
because all Hodge $(2,2)$-classes have $k=0$. It would be interesting to
know if the condition $k\geq k_m$ is sufficient for a Hodge $(2,2)$-class 
to be represented as $\beta(E)$ of a polystable bundle, when $k_m>0$.

\begin{remark}
Using the First Variation of Equation (\ref{eqn:main}), we can prove further: if $\beta$
is $\omega$-HYM and $\Spin$-rotable via $\gamma\in A_+$, then 
$(\beta_0 - 2k\, \omega^2)\wedge \gamma=0$ and $\beta_1\wedge \omega= \gamma\wedge\gamma'$ for
some $\gamma'\in A_-$.
\end{remark}

\medskip

We can rephrase Theorem \ref{thm:bogomolov} also in terms of 
$\Spin(7)$-geometry.

\begin{theorem}
 Suppose that $\beta$ is traceless $\omega$-HYM. Then
 $$
 k=\frac{1}{24} \max \{\frac{\beta \wedge (\omega')^2}{\vol} \, |\,  \omega'\in S(\bigwedge\nolimits^2_7)\},
$$
and those $\omega'$ where this minimum is achieved are exactly the K\"ahler forms such
that $\beta$ is traceless $\omega'$-HYM.
\end{theorem}

\begin{proof}
Let $\omega'\in S(\bigwedge^2_7)$. Then $\omega'=2\frac{\omega+\gamma}{|\omega+\gamma|}$, where
$\gamma\in A_+$. By Proposition \ref{prop:bogomolov}, $\beta\wedge\gamma^2 \leq 3k\, \omega^2\wedge
\gamma^2 = 6k \, |\gamma|^2$. Then 
 \begin{align*}
  \beta\wedge (\omega')^2 &=  \beta\wedge 4 \frac{(\omega+\gamma)^2}{|\omega+\gamma|^2} 
  = \frac{4}{4+|\gamma|^2} \, \beta\wedge (\omega^2+ 2\omega\wedge \gamma+ \gamma^2) \\ 
   & \leq  \frac{4}{4+|\gamma|^2} ( 24k + 6k  |\gamma|^2 ) \vol = 24 \, k\, \vol.
 \end{align*}
We defined $k'$ by $\beta\wedge (\omega')^2=24 \, k'\vol$. Therefore $k'\leq k$. If we have 
equality, then $\beta$ is $\Spin$-rotable via $\gamma$.
\end{proof}

\section{Examples} \label{sec:examples}

We want to see now some examples in which the new Bogomolov inequality 
is satisfied, and examples where it puts new
constraints for the existence of stable bundles.
Recall that $(X,\omega)$ is an abelian variety, $\beta\in H^4(X,\ZZ)$ is an
asymptotically $\omega$-HYM class, and $\beta\wedge\omega^2=k\, \omega^4$. Then 
  \begin{equation}\label{eqn:bogomolov}
  (\beta- 3k \omega^2 ) \wedge c \wedge \bar c\leq 0,
  \end{equation}
for any $c\in \bigwedge^{2,0}$.

\subsection{Complete intersections}
A basic case is that of degree $4$ classes $\beta$ which are complete intersections, i.e., 
product of divisors (rational degree $2$ classes). That is,  
bundles $E=\cO(D_1) \oplus \ldots \oplus \cO(D_r)$, where $D_1\cdot H=\ldots = D_r\cdot H$,
where $H$ denotes the polarization, and $\beta=\beta(E)$. 
This is the case where $E$ is $H$-polystable and
completely decomposable.

We start with a useful lemma.

\begin{lemma} \label{lem:alpha2}
 Let $\alpha\in \triangle^{1,1}_{prim}$. Then $\beta=-\alpha^2$ satisfies (\ref{eqn:bogomolov}).
\end{lemma}

\begin{proof}
Taking suitable coordinates, we can take $\omega$ to be standard and 
$c =dz_{12}+dz_{34}$ (see Remark \ref{rem:Spin6-transitive}). 
Write $\alpha=\sum a_{ij} dz_{i\bar{j}}$ and recall that
$a_{ji}=-\bar a_{ij}$ since $\alpha$ is a real form. In particular $a_{ii}$ are
purely imaginary. Denote $a=(a_{ij})$. As $\alpha$ is primitive, we have that 
$\sum a_{ii}=0$. Then $0=(\sum a_{ii})^2= \sum_{i\neq j} a_{ii}a_{jj}+ \sum a_{ii}^2$, hence 
$\sum_{i\neq j} a_{ii}a_{jj}= \sum |a_{ii}|^2$. Now we compute 
 $$
 \begin{aligned}
 24k\, \vol  &=\beta \wedge \omega^2 =- \alpha^2 \wedge \omega^2 \\
 &=- \sum_{a<b} a_{ij} a_{i'j'} 
  dz_{i\bar{j} i'\bar{j}'} \left(\frac{i}{2}\right)^2 2 dz_{a\bar a b \bar b} \\
 &=16 \, \frac12 \, \left( -\sum_{i\neq j}   a_{ij} a_{ji} + 
  \sum_{i \neq j} a_{ii}a_{jj} \right)\vol \\
 &=8  \, \left( \sum_{i\neq j}   |a_{ij}|^2 +  \sum_{i} |a_{ii}|^2 \right) \vol\\
 &= 8\left(\sum |a_{ij}|^2\right) \vol= 8 ||a||^2 \, \vol.
 \end{aligned}
 $$
An easy computation gives
 $$
 \beta \wedge c \wedge \bar c= 16(a_{13}a_{24} +a_{31}a_{42} 
+a_{14}a_{23} +a_{41}a_{32}
+ a_{34}a_{43}+a_{12}a_{21}+
a_{11}a_{22}+a_{33}a_{44}) \vol.
$$
Using that $|a_{13}a_{24}|\leq \frac12 (|a_{13}|^2+|a_{24}|^2)$, etc,
we have that 
 $$
 \beta \wedge c  \wedge \bar c  \leq 16 ||a||^2\vol =24\, k\, \vol= 
 3k\, |c |^2\,\vol = 3k \omega^2  \wedge c  \wedge \bar c ,
 $$
since $|c|=2\sqrt{2}$. This is the required inequality.
\end{proof}

Suposse that $E=E_1\oplus E_2$, where $E_1,E_2$ are holomorphic bundles of 
ranks $r_1,r_2$, and of the same
$\omega$-slope. Then
 $$
 \beta(E)=\beta(E_1)+ \beta(E_2)- \frac{r_1r_2}{2(r_1+r_2)}  
 \left(\frac{c_1(E_1)}{r_1}-\frac{c_1(E_2)}{r_2}\right)^2\, . 
 $$
If $E_1,E_2$ have the same $\omega$-slope then 
$$
\alpha=\frac{c_1(E_1)}{r_1}-\frac{c_1(E_2)}{r_2} \in \bigwedge\nolimits^{1,1}_{prim}
$$
 
Hence if $\beta(E_1)$ and $\beta(E_2)$ satisfy (\ref{eqn:bogomolov}), then $\beta(E)$
satisfies (\ref{eqn:bogomolov}), since $-\alpha^2$ does by Lemma \ref{eqn:bogomolov} and the sum of
classes which satisfy (\ref{eqn:bogomolov}) also satisfy it.

\begin{corollary}
 Let $E=\cO(D_1)\oplus \ldots \oplus \cO(D_r)$ be a rank $r$ polystable bundle which
is a direct sum of line bundles of the same $\omega$-slope.
Then $\beta(E)$ satisfies (\ref{eqn:bogomolov}).
\end{corollary}

\subsection{Diagonal property for self-products} \label{subsec:diagonal}
Consider the  Jacobian $J$ of a (generic) genus $2$ curve, so $J$ is a $2$-dimensional
abelian variety which is principally polarised by some divisor $\theta$. By 
\cite{Debarre}, there is a vector bundle $F \to J$ with $c_1(F)= \theta$, $c_2(F)=p_0$, the class of a point.
In particular, $\beta(F)=c_2-\frac14 c_1^2=\frac34 p_0$.

Consider the bundle $E \to X=J \x J$, which is the pull-back of $F$ under the 
 map $\sigma: J\x J \to J$, $\sigma(x_1,x_2)=x_1-x_2$. It has $\beta(F)=  \frac34 \Delta$,
where $\Delta$ is the diagonal. Take the K\"ahler class $\omega=\sqrt{2}(\theta_1+\theta_2)$, where
$\theta_i=pr_i^* \theta$. Here we have normalized so that $\omega^4=24$ and the total volume to $1$. So
 $$
 k=\frac1{24} \beta \wedge\omega^2 =\frac1{32} \Delta \wedge \omega^2= \frac12\, ,
 $$
since $\Delta \wedge \omega^2=(2\sqrt{2}\theta)^2=8$.
If we write, in a suitable basis, $\omega=\frac{i}{2}\sum dz_{j\bar{j}}$, where
$(z_1,z_2)$ are the coordinates of the first factor $J$, and 
$(z_3,z_4)$ are the coordinates of the second factor $J$, then
we can take 
 $$
 c_1 =dz_{12}+ dz_{34}.
 $$
Clearly
 \begin{align*}
 &\beta \wedge c_1 \wedge \bar{c}_1 =\frac34 \Delta \wedge c_1 \wedge \bar{c}_1 = 
 \frac34 (2dz_{12}) \wedge (2d \bar{z}_{12}) = \frac34\, 16=12, \\
 &\omega^2 \wedge c_1 \wedge \bar{c}_1 = 16.
 \end{align*}
So we have 
 $(\beta- 3 k \omega^2) \wedge c_1 \wedge\bar{c}_1=  0$, which means that $\beta$ is $\Spin$-rotable.
Using the basis in Remark \ref{rem:aspa}, we check that $\beta\wedge c_j \wedge \bar{c}_j 
\leq 12$, for $j=2,\ldots 6$, agreeing with the Bogomolov inequality.

Note that the class $D=\Delta-\epsilon \, \omega^2$, $\epsilon>0$ small, does not satisfy
(\ref{eqn:bogomolov}), so it cannot be $\beta(E)$ for any polystable bundle $E$.
Even for any large $\ell>0$, $\ell D$ also cannot be $\beta(E)$ for a polystable bundle 
$E$, although the associated value $k=\frac{1}{24} \ell D\cup [\omega]^2$ is as large as we want.

\medskip

  The above situation is isomorphic
to taking $J\x \{p_0\}\subset J\x J$. More generally, we can consider 
 $X=Y \times Y'$, where $Y, Y'$ are complex $2$-tori, and $F=Y\x \{p_0\}$. Then
 the class  
  $$
   \beta=dz_{12\bar1\bar2}
   $$
is $\Spin$-rotable via  
  $$
  c\in \langle dz_{12}+dz_{34} , i\, dz_{12}-i\, dz_{34} \rangle\, .
  $$
We shall see in Subsection \ref{subsec:rot-prod} what is the rotated $4$-torus.

\subsection{Weil classes in Weil complex tori}
Weil complex tori are complex tori whose endomorphism
ring contains a purely imaginary quadratic field,  
$\QQ[\sqrt{-d}] \subset \End (X)$, $d>0$ square-free.
Let us describe them explicitly. Consider the ring
$L=\ZZ[\sqrt{-d}]$ and the field $K=\QQ [\sqrt{-d}]$.
There is a natural map $\varphi:L\to L$, $\varphi(x)=\sqrt{-d}\, x$.

Consider the lattice $\Lambda = L^4$, and let $V=\Lambda \otimes \RR \cong \RR^8$.
The map $I(x)= \frac{1}{\sqrt{d}} \, \varphi (x)$, $I:V\to V$, defines a
complex structure. So $(V,I)\cong (\CC^4,i)$. Now choose
two $I$-complex subspaces of dimension $2$, 
 $$
  V=W_+ \oplus W_- \, ,
  $$
and define
 $$
  J:V \to V, \qquad  J|_{W_+}=I, \quad J|_{W_-}=-I\, .
 $$
Then $(V,J)$ is a complex vector space, $\Lambda\subset V$, and let
 $$
 X=(V,J)/\Lambda\, .
 $$

This is a complex $4$-torus. Note that $IJ=JI$, so $\varphi:X\to X$
is holomorphic, and $\varphi^2=-d\, \text{Id}$.
Hence $\ZZ[\sqrt{-d}]\subset \End (X)$. This is called a \emph{Weil complex torus}.
For generic choice of $W_\pm$, we have $\ZZ[\sqrt{-d}]=\End(X)$.

The Weil classes are the $(2,2)$-classes lying in
   \begin{align*}
    K \cong & \bigwedge\nolimits^4_K (\Lambda\ox \QQ) \subset
    \bigwedge\nolimits^{4,0}_I V  = \bigwedge\nolimits^{4,0}_I (W_+\oplus W_-) =\\
     &=\bigwedge\nolimits^{2,0}_I  W_+ \otimes \bigwedge\nolimits^{2,0}_I {W}_- 
    =\bigwedge\nolimits^{2,0}_J  W_+ \otimes \bigwedge\nolimits^{0,2}_J {W}_-
    \subset \bigwedge\nolimits^{2,2}_J V\, .
  \end{align*}
These are Hodge classes (rational classes of pure type).
We note that the space of Weil classes is a rational vector space of dimension $2$.
The generic Weil complex torus does not have more Hodge classes \cite{Moonen-Zarhin}.

We may put many K\"ahler forms on $X$. Note that
 $$
  \bigwedge\nolimits^{1,1}_I V= \bigwedge\nolimits^{1,1} W_+ \oplus \bigwedge\nolimits^{1,1} W_- \
        \oplus \bigwedge\nolimits^{1,0} W_+ \bigwedge\nolimits^{0,1} W_-
        \oplus \bigwedge\nolimits^{0,1} W_+ \bigwedge\nolimits^{1,0} W_-
  $$
If we consider K\"ahler forms so that $W_+ \perp W_-$, then $\omega$ lies in
  \begin{equation}\label{ddde}
  \bigwedge\nolimits^{1,1} W_+ \oplus \bigwedge\nolimits^{1,1} W_- \subset \bigwedge\nolimits^{1,1}_I V\, .
  \end{equation}
We want the Riemannian metric $g$ to be positive-definite for $X$, i.e.
$g=g_+ + g_-$, where $g_\pm$ is positive definite on $W_\pm$. As $g(x,Jy)=\omega(x,y)$,
and $J_\pm=\pm I$ on $W_\pm$, we see that $\omega$ defines an \emph{indefinite}
semi-riemannian metric $\tilde g=g_+ - g_-$ for $(V,I)$.
Said otherwise, we have complex coordinates $(w_1,w_2,w_3,w_4)$ for $(V,I)$
such that $\tilde h= \frac{i}2 (dw_{1\bar1}+dw_{2\bar2}-dw_{3\bar3}-dw_{4\bar4})$
is an indefinite K\"ahler form, $\tilde h=\tilde g+ i \omega$, $W_+$ has coordinates
$(w_1,w_2)$ and $W_-$ has coordinates $(w_3,w_4)$. The complex coordinates for $(V,J)$
are $(z_1,z_2,z_3,z_4)=(w_1,w_2,\bar w_3,\bar w_4)$ and the induced
K\"ahler form is $h= \frac{i}2 (dz_{1\bar1}+dz_{2\bar2}+dz_{3\bar3}+dz_{4\bar4})$.

Take now a Weil class
 $\beta_0 = 2\,\Re(dw_{1234}) \in \bigwedge_K^4 (K^4)$. In coordinates $(z_1,z_2,z_3,z_4)=
(w_1,w_2,\bar w_3,\bar w_4)$, 
 $\beta_0= 2 \, \Re(dz_{12\bar3\bar4})= dz_{12\bar3\bar4}+dz_{\bar1\bar234}$, 
which is rational and of $J$-type $(2,2)$, up to normalization (that we shall ignore).
Recall that $W_+=\la z_1,z_2 \ra$, $W_-=\la z_3,z_4 \ra$.
Using the basis of $A_+$ in Remark \ref{rem:aspa}, one checks easily that
 \begin{equation}\label{eqn:mm2}
 k_m(\beta_0)=1.
 \end{equation}
The maximum value of $\beta_0 \wedge c \wedge \bar c$ is $32$ and it 
is achieved for $c=c_1=dz_{12}+dz_{34}$. 

So if we consider 
\begin{equation} \label{eqn:mmm} 
 \begin{aligned}
 \beta_0 &=dz_{12\bar3\bar4}+dz_{\bar1\bar234}, \\
 \omega &=\frac{i}{2} (dz_{1\bar1}+dz_{2\bar2}+dz_{3\bar3}+dz_{4\bar4}), \\
 \beta &= \beta_0 + k \, \omega^2, \\
 k &= 1, 
 \end{aligned}
 \end{equation}
the $\Spin$-rotation equation (cf.\ Equation (\ref{eqn:main})) 
 $$
 (\beta_0 - 2k\,\omega^2) \wedge c \wedge \bar{c} = 0
 $$
has solution
 $$
  c  = dz_{12}+dz_{34}\, .
 $$

Summarizing, 
 \begin{equation} \label{eqn:alfa}
 \beta=(dz_{12\bar3\bar4}+dz_{\bar1\bar234}) + \omega^2
 \end{equation}
is a $\Spin$-rotable class, where the rotation is given by $c  = dz_{12}+dz_{34}$.
We shall see later that $\Spin$-rotation produces a modification 
of the complex structure on $W_+=\la z_1,z_2 \ra$
and on $W_-=\la z_3,z_4 \ra$, keeping them orthogonal.

\section{Rotating in $2$ dimensions}\label{sec:2-dim}

There is an analogue of the $\Spin$-rotation in $2$ complex dimensions, that is
worth to review. It is closely related to the twistorial construction.
We mention that this idea has been used by M. Toma \cite{Toma} to produce 
stable bundles on complex $2$-tori.

Let $X=V/\Lambda$ be a $2$-dimensional complex torus, with a K\"ahler form $\omega$,
and complex structure $J$. Let $\Pi$ be its period matrix, that is the $2\x 4$-matrix
whose columns are the vectors of a basis of the lattice $\Lambda$, with respect to
some complex coordinates of $V\cong \CC^2$. This is well defined up to the action
of a matrix of $\GL(2,\CC)$ on the left and a matrix of $\GL(4,\ZZ)$ on the right.

In $2$ dimensions, we have the inclusion of groups
 $U(2)\subset SO(4)$, and the quotient
 $$
 SO(4)/U(2)=S^2=S(\bigwedge\nolimits^2_+)
$$ 
parametrizes complex structures on $X$ compatible with the metric. Here
 $$
\bigwedge\nolimits^2_+=\triangle^{2,0} \oplus \langle \omega\rangle, \qquad
\bigwedge\nolimits^2_-=\triangle^{1,1}_{prim} \, .
 $$

We can rotate the complex structure by taking $r,s\in \RR$, $r^2+s^2=1$ and
considering the new K\"ahler form
 $$
 \omega'=r\, \omega+s \, \gamma,
 $$
where $\gamma \in \triangle^{2,0}$, $|\gamma|=|\omega|=\sqrt{2}$. This defines a new complex
structure $J'$ by the equation $g(x,y)=\omega'(x,J'y)$.

Let us find explicitly the new abelian variety $X'=(V,J')/\Lambda$.
For a suitable choice of complex coordinates $(z_1,z_2)$ for $(V,J)=\CC^2$, we may write
  \begin{align*}
  \omega &= \frac{i}{2} (dz_{1\bar1}+dz_{2\bar2}), \\
  \gamma &= \frac12 (dz_{12}+dz_{\bar1\bar2}),
  \end{align*} 
where we abbreviate $dz_{12}=dz_1\wedge dz_2$, $dz_{1\bar1}=dz_1\wedge d\bar z_1$, etc.

We introduce real coordinates $z_1=x_1+iy_1$, $z_2=x_2+iy_2$, so that
  \begin{align*}
  \omega &= dx_1\wedge dy_1+ dx_2\wedge dy_2, \\
  \gamma &= dx_1\wedge dx_2 - dy_1\wedge dy_2.
  \end{align*} 
Therefore
 $$
 \omega'=r\,dx_1\wedge dy_1+ r\,dx_2\wedge dy_2 + s\,dx_1\wedge dx_2 - s \, dy_1\wedge dy_2.
 $$
Using the equation $g(x,y)=\omega'(x,J'y)$, we easily compute the matrix of $J'$ in this basis:
 $$
 J'=\left( \begin{array}{cccc} 0 & -r & -s & 0 \\ r & 0 & 0 & s \\ s & 0 & 0 & -r \\ 0 & -s & r & 0 
 \end{array} \right).
 $$
We want to put new coordinates so that $J'$ becomes the standard complex structure. These
coordinates are given by
 $$
 \left( \begin{array}{c} x_1' \\ y_1' \\ x_2' \\ y_2' \end{array} \right)
 = M  \left( \begin{array}{c} x_1 \\ y_1 \\ x_2 \\ y_2 \end{array} \right), \ \text{where }
 M=\left( \begin{array}{cccc} 1 & 0 & 0 & 0 \\ 0 & r & s & 0 \\ 0 & -s & r & 0 \\ 0 & 0 & 0 & 1
 \end{array} \right).
 $$
The new complex coordinates are $z_1'=x_1'+iy_1'$, $z_2'=x_2'+iy_2'$.

Conjugating by $M$, we see that the new complex torus is $X' \cong \CC^2/\Lambda'$, where
the matrix period corresponding to $\Lambda'=M\, \Lambda$ is 
 $$
 \Pi'= M\, \Pi 
 $$
(writing the matrix $\Pi$ as a $4\x 4$-real matrix by putting the real and imaginary
parts of each complex entry vertically).

%

\subsection{Rotating a product of two $2$-tori} \label{subsec:rot-prod}
Let us start by looking at a simple case of $\Spin$-rotation of a complex $4$-torus
which is the product of two complex $2$-tori.
Let $X=Y_1 \times Y_2$, where $Y_1, Y_2$ are complex $2$-tori. Put
a product metric, and orthonormal complex coordinates $(z_1,z_2,z_3,z_4)$,
where $(z_1,z_2)$ are coordinates of $Y_1$ and $(z_3,z_4)$ are coordinates
of $Y_2$. 
Consider the class
  $$
   \beta=dz_{12\bar1\bar2} \in H^4(X).
   $$
This is the Poincar\'e dual of a multiple of $F=Y_1\x \{pt\}\subset X$, which is
an algebraic class. Up to a multiple, $\beta$ is a rational class. By \cite{Debarre},
there is a stable rank $2$ bundle $E\to Y_2$ with $c_1=0$, $c_2=\{pt\}$. Pulling it
back to $X$, we see that $\beta$ is $\omega$-HYM (up to a multiple, which we shall 
ignore).

We compute the value 
 $$
 k=\frac{\beta\wedge \omega^2}{\omega^4}=\frac{-\frac12 dz_{12\bar1\bar2}
 \wedge dz_{3\bar34\bar4}}{\frac{24}{16} dz_{1234\bar1\bar2\bar3\bar4}}= \frac13\,.
$$
Any rotation parameter
  $$
  c \in \langle dz_{12}+dz_{34} , i\, dz_{12}-i\, dz_{34} \rangle
  $$
satisfies the rotation equation (\ref{eqn:main}): take $c =
\frac12 (w \, dz_{12}+\bar{w} dz_{34})$, where $w\in \CC$.
Then $|c|=\sqrt{2} \, |w|$, and we compute
  $$
  (\beta-3 k \, \omega^2)\wedge c \wedge \bar{c}= 
   \frac14 |w|^2 \,  16 \,\vol - 3 k\, 2(\sqrt{2}\,|w|)^2 \vol =0.
  $$

Applying the $\Spin$-rotation, we get a
new complex structure $J'$ given by the same Riemannian metric and the K\"ahler form
 $$
 \omega'= \frac{1}{\sqrt{1+|w|^2}}\left(
 \frac{i}{2}(dz_{1\bar1}+dz_{2\bar2}+dz_{3\bar3}+dz_{4\bar4}) + \frac12(w
 \, dz_{12}+ \bar{w} dz_{\bar1\bar2} +\bar{w} dz_{34}+ w\,
  dz_{\bar3\bar4}) \right).
  $$
Clearly, $\omega'=\omega_1'+\omega_2'$ where
$\omega_1'=\frac{1}{\sqrt{1+|w|^2}}(\frac{i}{2}(dz_{1\bar1}+dz_{2\bar2})+ \frac12(w
 \, dz_{12}+ \bar{w} dz_{\bar1\bar2}))$ and 
$\omega_2'=\frac{1}{\sqrt{1+|w|^2}}(\frac{i}{2}(dz_{3\bar3}+dz_{4\bar4}) +\frac12( 
\bar{w} dz_{34}+ w\, dz_{\bar3\bar4}))$. Therefore $X'=
Y_1' \times Y_2'$, where
$Y_1'$ corresponds to $(Y_1,\omega_1')$ and $Y_2'$ corresponds to
$(Y_2,\omega_2')$.
 So the rotated torus keeps being a product of two $2$-tori.

\section{$\Spin$-rotation of a Weil abelian variety} \label{sec:rotating-Weil}

Now we want to give explicitly a non-trivial example of a $\Spin$-rotation of
a complex $4$-torus, starting with a Weil abelian variety $X$ which is a product
of two Weil abelian surfaces. The $\Spin$-rotated complex torus turns out to
be another Weil abelian variety $X'$ which is not decomposable.
Moreover, $X'$ will be, from the arithmetic point of view, very different from
the starting torus $X$.

As this is a long discussion, we shall divide it in steps for the convenience of the
reader.

\subsection{General set-up}
We start by writing down explicitly the period matrix of a Weil abelian variety. 
Consider $K=\QQ[\sqrt{-d}]$, $d>0$ a square-free integer. Write $(v_1,v_2,v_3,v_4)$ for 
the standard coordinates of $V=K^4\otimes \RR$. 
We consider the indefinite hermitian form $\tilde h=\frac{i}{2}(dv_{1\bar1}+dv_{2\bar2}-
dv_{3\bar3}-dv_{4\bar4})$. 

Take a basis for the subspaces $W_\pm$, which we write as the columns of the matrix 
  $$
  \left( \begin{array}{cccc} 1 & 0 & a' & e' \\
  0 & 1 & b' & f' \\
  a & e & 1 & 0 \\
  b & f & 0 & 1
  \end{array} \right)
  $$
(the first two columns are the basis of $W_+$ and the remaining two are the basis of $W_-$).
We call this matrix the \emph{defining matrix} of the Weil torus.
For $X$ to be an abelian four-fold, we have to take $W_+\perp W_-$, w.r.t.\ $\tilde h$. In this case,
the defining matrix is of the form
  \begin{equation}\label{eqn:AA}
  \left( \begin{array}{cccc} 1 & 0 & \bar a & \bar b \\
  0 & 1 & \bar e & \bar f \\
  a & e & 1 & 0 \\
  b & f & 0 & 1
  \end{array} \right).
  \end{equation}
Moreover, note that $I-AA^*>0$,
where $A=  \left( \begin{array}{cc} a & e \\ b & f  \end{array} \right)$.
If we change to coordinates $(w'_1,w'_2,w'_3,w'_4)$ in which the 
subspaces $W_\pm$ are the standard ones (by multiplying by the inverse of
the matrix (\ref{eqn:AA})), the lattice becomes generated (over $L=\ZZ[\sqrt{-d}]$)
by the columns of 
 $$
 \left(\begin{array}{cc} (I-A^*A)^{-1} &0 \\ 0 & (I-AA^*)^{-1} \end{array} \right)\, 
 \left(\begin{array}{cc} I & -A^* \\ -A & I \end{array} \right)
 $$
The condition $W_+\perp W_-$ gives that the metric is rewritten 
as $\tilde h=\frac{i}{2} (dw'_{1\bar1}+dw'_{2\bar2}-
dw'_{3\bar3}-dw'_{4\bar4})$. By definition, $\tilde h$ is defined over $K$,
so it takes values in $L=\ZZ[\sqrt{-d}]$ on $\Lambda$.
Writing $\tilde h=\tilde g+ i \omega$, we have that
$\omega=\Im(\tilde h)$ takes values in $\delta\,\ZZ$, $\delta=\sqrt{d}$, 
on the lattice $\Lambda$.

The lattice of a complex torus is defined up to multiplication by
a matrix of $\GL(4,\CC)$ on the left. If we multiply
by a matrix of the type $ \left(\begin{array}{cc} P & 0 \\ 0 & Q \end{array} \right)$,  we ensure
that the subspaces $W_\pm$ are still the standard ones, but can arrange the lattice to be 
generated by the columns of 
  $$
  \left( \begin{array}{cccc} 1 & 0 & -\bar a & -\bar b \\
  0 & 1 & -\bar e & -\bar f \\
  -a & -e & 1 & 0 \\
  -b & -f & 0 & 1
  \end{array} \right)
  $$
and these elements multiplied by $\sqrt{-d}=i\delta$. 
Let $(w_1,w_2,w_3,w_4)$ be these new coordinates.
To go to the complex structure $J$, we need to conjugate the complex
structure on $W_-$. So we take $z_3=\bar w_3$, $z_4=\bar w_4$.
The lattice becomes
  \begin{equation}\label{eqn:AA2}
  \left( \begin{array}{cccc} 1 & 0 & -\bar a & -\bar b \\
  0 & 1 & -\bar e & -\bar f \\
  -\bar a & -\bar e & 1 & 0 \\
  -\bar b & -\bar f & 0 & 1
  \end{array} \right).
  \end{equation}
and the $\varphi$-transforms of those, i.e.\ the period matrix is
  $$
  \Pi= \left( \begin{array}{cccccccc} 
  1 & i\delta & 0 & 0 & -\bar a & -i\delta\bar a & -\bar b & -i\delta\bar b \\
  0 & 0 & 1 & i\delta & -\bar e & -i\delta\bar e & -\bar f & -i\delta\bar f \\
  -\bar a & i\delta\bar a & -\bar e & i\delta\bar e & 1 & -i\delta & 0 & 0 \\
  -\bar b & i\delta\bar b & -\bar f & i\delta\bar f & 0 & 0 & 1 & -i\delta
  \end{array} \right),
$$
where $\varphi=\left( \begin{array}{cccc} i\delta & 0 & 0 &0 \\ 0 & i\delta &0 &0\\
0 & 0& -i\delta &0 \\ 0& 0& 0 & -i\delta \end{array} \right)$. This is the general
form of the period matrix of a Weil abelian variety.
We say that (\ref{eqn:AA2}) is the \emph{reduced period matrix} of $X$.

\smallskip

To continue, we shall put an extra condition to ensure that the $\Spin$-rotated
torus is an abelian variety. 
If $\omega$ is the K\"ahler form of $X$, then the K\"ahler form of the $\Spin$-rotated
torus is
 $$
 \omega'=2 \frac{\omega+\gamma}{|\omega+\gamma|}= r\omega+ s\, \gamma,
 $$
where $r,s\in \RR$ satisfy $r^2+s^2=1$,
where we have taken $\gamma\in A_+$ with $|\gamma|=\sqrt{2}$ 
(see (\ref{eqn:omega'}) for notations).

In the coordinates $(z_1,z_2,z_3,z_4)$ of $X$, we have the
following choices obtained in (\ref{eqn:mmm}) for performing a 
$\Spin$-rotation,
 \begin{equation}\label{eqn:alfa}
 \begin{aligned} %
 \beta &=(dz_{12\bar3\bar4}+dz_{\bar1\bar234}) + \omega^2, \\
  \omega &= \frac{i}{2} (dz_{1\bar1}+dz_{2\bar2}+dz_{3\bar3}+dz_{4\bar4}), \\
  \gamma &= \frac12 (dz_{12}+dz_{34}+dz_{\bar1\bar2}+dz_{\bar3\bar4}),
 \end{aligned}
 \end{equation}
In order for $X'$ to be an abelian variety, we shall require that $\omega'$ be a rational class,
which can be guaranteed if $\gamma=\Re(dz_{12}+dz_{34})$ is a rational class and
$r,s\in \QQ$, with $r^2+s^2=1$.
For achieving this, take the holomorphic symplectic form $c=dv_{12}+dv_{34}$ for $(V/\Lambda,I)$.
It takes values in $\ZZ[\sqrt{-d}]$ on $L$. Now take $W_+ \perp W_-$ w.r.t.\ $c$. In
terms of the matrix (\ref{eqn:AA}), this puts the extra condition
$f=-\bar a, e=\bar b$, i.e.,
 $$
 A=\left(\begin{array}{cc} a & -\bar b\\ b & - \bar a\end{array}\right).
 $$
In this case $A$ is an orthogonal
matrix (i.e.\ it is in $\RR_+\cdot \U(2)$), so the metric $\tilde h$
only differs by a factor when changing from coordinates $(w_1',w_2',w_3',w_4')$ to
$(w_1,w_2,w_3,w_4)$. Absorbing this factor 
$\nu=\det(I-AA^*)=1-|a|^2-|b|^2$ in the metric, we have that $\tilde h=
\frac{i}{2} (dw_{1\bar1}+dw_{2\bar2}-dw_{3\bar3}-dw_{4\bar4})$, and 
$\omega=\Im (\tilde h)$ takes values in $\nu \delta \ZZ$.

Now the class $c$ is rewritten as $c=dw_{12}+dw_{34}=(1-|a|^2-|b|^2) (dv_{12}+dv_{34})$. It takes values in 
$\nu\, L$. As $\gamma=\Re (c)=\Re(dw_{12}+dw_{34})=
\Re(dz_{12}+dz_{34})$, under $(z_1,z_2,z_3,z_4)=(w_1,w_2,\bar w_3,\bar w_4)$, we have that
$\gamma$ takes values in $\nu\ZZ$.
So $\omega'=r\omega+s\gamma$ is rational (up to a multiple) when  
$r^2 + s^2=1$,
where $s=\delta \hat s$, $r, \hat s\in \QQ$. This is equivalent to
$r^2+d \hat s^2=1$, $r,\hat s\in \QQ$.

\subsection{Period matrix of a $\Spin$-rotated torus}
Summarizing, our starting Weil abelian torus has
reduced period matrix (\ref{eqn:AA2}) 
 $$
  \left( \begin{array}{cccc} 1 & 0 & -\bar a & -\bar b \\
  0 & 1 & -b & a \\
  -\bar a & -b & 1 & 0 \\
  -\bar b & a & 0 & 1
  \end{array} \right).
  $$
In real coordinates, the period matrix is
 $$
   \Pi=\left( \begin{array}{cccccccc}
 1&  0&  0&  0&  -a_1&  -a_2 \delta &  -b_1&  -b_2 \delta \\  
 0&  \delta &  0&  0&  a_2&  -a_1 \delta &  b_2&  -b_1 \delta \\   
 0&  0&  1&  0&  -b_1&  b_2 \delta &  a_1&  -a_2 \delta \\  
 0&  0&  0&  \delta &  -b_2&  -b_1 \delta &  a_2&  a_1 \delta \\   
 -a_1&  a_2 \delta &  -b_1&  -b_2 \delta  &  1&  0&  0&  0\\  
 a_2&  a_1 \delta &  -b_2&  b_1 \delta &  0&  -\delta &  0&  0\\    
-b_1&  b_2 \delta &  a_1&  a_2 \delta &  0&  0&  1&  0\\  
 b_2&  b_1 \delta &  a_2&  -a_1 \delta &  0&  0&  0&  -\delta   
\end{array} \right).
  $$

If  $z_1=x_1+iy_1$, $z_2=x_2+iy_2$, $z_3=x_3+iy_3$, $z_4=x_4+iy_4$ are the complex
coordinates of $X$,
and  $z_1'=x_1'+iy_1'$, $z_2'=x_2'+iy_2'$, $z_3'=x_3'+iy_3'$, $z_4'=x_4'+iy_4'$ are
the complex coordinates of the $\Spin$-rotated torus $X'$, the discussion in 
Section \ref{sec:2-dim} implies that they are related by
  $$
 \left( \begin{array}{c} x_1' \\ y_1' \\ x_2' \\ y_2' \end{array} \right)
 = M  \left( \begin{array}{c} x_1 \\ y_1 \\ x_2 \\ y_2 \end{array} \right), \
  \text{and} \left( \begin{array}{c} x_3' \\ y_3' \\ x_4' \\ y_4' \end{array} \right)
 = M  \left( \begin{array}{c} x_3 \\ y_3 \\ x_4 \\ y_4 \end{array} \right), 
\ \text{where }
 M=\left( \begin{array}{cccc} 1 & 0 & 0 & 0 \\ 0 & r & s & 0 \\ 0 & -s & r & 0 \\ 0 & 0 & 0 & 1
 \end{array} \right).
 $$

\begin{remark}
The change of variables relating $(x_j,y_j)$ and
$(x_j',y_j')$ can be explicitly used to check the formula
 $$
 \omega'=\frac{i}{2} (dz'_{1\bar1}+dz'_{2\bar2}+dz'_{3\bar3}+dz'_{4\bar4}),
 $$
and also to rewrite the class (\ref{eqn:alfa}) in the new coordinates. The result is 
 $$
\beta = (dz'_{12\bar3\bar4}+dz'_{\bar1\bar234}) + (\omega')^2\, .
 $$
Note that $\beta$ is type $(2,2)$ w.r.t.\ $J'$ as expected by Theorem \ref{thm:rotation}.
\end{remark}

\medskip

The rotated Weil abelian variety $X'$ has period matrix $\Pi'$ obtained by multiplying
$\Pi$ by 
 \begin{equation}\label{eqn:29A0}
  \tilde M= \left( \begin{array}{cc} M & 0 \\  0 & M \end{array}\right).
 \end{equation}
That is,
  $$
\Pi' = \left( \begin{array}{cccc}
 1&  0&  0&  0 \\ 
 0&  \delta r &  s&  0 \\ 
 0&  -\delta s&  r&  0 \\ 
 0&  0&  0&  \delta \\ 
 -a_1&  a_2 \delta &  -b_1&  -b_2 \delta  \\ 
 a_2 r -b_1 s&  (a_1r+b_2s) \delta &  -b_2r + a_1s &  (b_1r+a_2s) \delta \\ 
-b_1r-a_2s& ( b_2r-a_1s) \delta &  a_1r+b_2s&  (a_2 r-b_1s)\delta \\ 
 b_2&  b_1 \delta &  a_2&  -a_1 \delta 
\end{array}  \right. \hspace{4cm} \, 
  $$
    $$
   \hspace{4cm} \left. \begin{array}{cccc}
  -a_1&  -a_2 \delta &  -b_1&  -b_2 \delta \\  
a_2 r- b_1s &  (-a_1 r+b_2 s) \delta &  b_2r+a_1s&  (-b_1r -a_2s)\delta \\   
  -a_2s-b_1r &  (a_1s+b_2r) \delta & - b_2s + a_1r&  (b_1s-a_2r) \delta \\  
 -b_2&  -b_1 \delta &  a_2&  a_1 \delta \\   
1&  0&  0&  0\\  
0&  -\delta r & s&  0\\    
  0&  \delta s&  r&  0\\  
0&  0&  0&  -\delta  \end{array} \right).
  $$
In complex terms, it is 
  $$
  \left( \begin{array}{cccc}
 1&  \delta r i&  si&  0 \\ 
    0&  -\delta s&  r&  \delta i \\ 
 -a_1+ ( a_2 r -b_1 s) i &  a_2 \delta  +(a_1r+b_2s) \delta i &  -b_1+ (-b_2r + a_1s) i&  
 -b_2 \delta + (b_1r+a_2s) \delta i  \\ 
-b_1r-a_2s+ b_2 i& ( b_2r-a_1s) \delta+  b_1 \delta i &  a_1r+b_2s+ a_2 i&  (a_2 r-b_1s)\delta -a_1 \delta  i
\end{array}  \right. \hfill \,
  $$
  $$
  \hfill  \left. \begin{array}{cccc}
-a_1+(a_2 r- b_1s)i&  -a_2 \delta + (-a_1 r+b_2 s) \delta i &  -b_1 +(b_2r+a_1s)i 
&  -b_2 \delta +(-b_1r -a_2s)\delta i \\
-a_2s-b_1r  -b_2 i&  (a_1s+b_2r) \delta  -b_1 \delta i & - b_2s + a_1r+ a_2 i 
&  (b_1s-a_2r) \delta + a_1 \delta i\\   
  1&  -\delta r i &  s i &  0\\  
 0&  \delta s&  r&  -\delta i  
\end{array} \right).
  $$

\subsection{Our starting abelian variety}
To perform explicitly the $\Spin$-rotation, we are going to consider the case
where the reduced  period matrix of $X$ is 
 \begin{equation}\label{eqn:red-per-mat}
  \left( \begin{array}{cccc} 1 & 0 & -\bar a & 0\\ 0 & 1 & 0 & a 
   \\ -\bar a & 0 & 1 & 0 \\ 0 & a & 0 & 1 \end{array}\right).
  \end{equation}
This means that $X=Y_1\x Y_2$, where $Y_1,Y_2$ are two abelian surfaces of
Weil type. The first abelian surface $Y_1$ corresponds to coordinates $(z_1,z_3)$
and the second one $Y_2$, to coordinates $(z_2,z_4)$.

\begin{lemma}\label{lem:product-non-trivial}
The two surfaces $Y_1,Y_2$ are non-isomorphic, for generic $a\in \CC$, $|a|<1$. 
\end{lemma}

\begin{proof}
 $Y_1$ has defining matrix $\left(\begin{array}{cc} 1 & a \\ a & 1\end{array}\right)$ 
and $Y_2$ is given by $\left(\begin{array}{cc} 1 & -\bar a \\ -\bar a & 1\end{array}\right)$.
The endomorphism rings are $\End(Y_1)=\End(Y_2)=\QQ[\sqrt{-d}]$.

Any possible isomorphism $\psi:Y_1\to Y_2$ should interchange the endomorphisms 
$\varphi$, up to sign. In particular, $\psi$ either preserves or permutes 
the two eigenspaces of $\varphi$. Suppose the first case (the second one is analogous). 
Let $B$ be the matrix of $\psi$ w.r.t.\ the coordinates of $K^2$. Then 
 $$
 \left(\begin{array}{cc} 1 & -\bar a \\ -\bar a & 1\end{array}\right)
 \left(\begin{array}{cc} \epsilon_1 & 0 \\ 0 & \epsilon_2\end{array}\right)
 = B \left(\begin{array}{cc} 1 & a \\  a & 1\end{array}\right).
$$
where $B$ is a matrix with coefficients in $K=\QQ[\sqrt{-d}]$, $\epsilon_1,\epsilon_2\in \CC$.
Hence we have an equation
 $$
 -\bar a= \frac{\gamma+\delta a}{\alpha+\beta a},
 $$
for some $\alpha,\beta,\gamma,\delta \in K$. This cannot happen for general $a$.
\end{proof}


\medskip

For the Weil abelian surface $Y_1$, we have a Weil class, which is of the form
$\alpha_1=2\,\Re(dw_{13})=2 \, \Re(dz_{1\bar3})=dz_{1\bar3}+dz_{\bar13}$
(recall that $Y_1$ has coordinates $(z_1,z_3)$ and $Y_2$ has coordinates $(z_2,z_4)$). In this
situation, the Weil class is a rational $(1,1)$-class. By the Lefschetz theorem,
there is a (rational) divisor $D_1$ and a corresponding holomorphic line bundle
$\cL_1=\cO(D_1)$ over $Y_1$ with $c_1(\cL_1)=[D_1]=\alpha_1$ (after multiplying
by a large integer if necessary). Note that $\alpha_1$ is primitive, so 
$\alpha_1\wedge \omega_1=0$.

For $Y_2$ we have a second Weil class $\alpha_2=-dz_{2\bar4}-dz_{\bar24}$ and a
divisor $D_2$ and line bundle 
 $\cL_2=\cO(D_2)$ with $c_1(\cL_2)=[D_2]=\alpha_2$. The rank $2$ bundle 
 $$
 E=\cL_1 \oplus \cL_2
 $$
is polystable, since the degrees of both summands are zero
(w.r.t.\ the K\"ahler form $\omega=\omega_1+\omega_2$). 
We compute 
  $$
 4\beta(E)= -(D_1-D_2)^2 = -2 dz_{1\bar13\bar3}-2dz_{2\bar24\bar4} 
  + 2dz_{12\bar3\bar4} + 2dz_{\bar1\bar234} + 2dz_{1\bar2\bar34} + 2dz_{\bar123\bar4}.
  $$
Starting with $\alpha_1'=-2\Im(dz_{1\bar3})=i\,dz_{1\bar3}-i\,dz_{\bar13}$ and 
$\alpha_2'=i\,dz_{2\bar4}-i\,dz_{\bar24}$,
we get a bundle $E'=\cL_1'\oplus \cL_2'$ such that
  $$
 4\beta(E')= -2 dz_{1\bar13\bar3}-2dz_{2\bar24\bar4} 
  + 2dz_{12\bar3\bar4} + 2dz_{\bar1\bar234} - 2dz_{1\bar2\bar34} - 2dz_{\bar123\bar4}.
  $$
The direct sum $E\oplus E'$ has 
  $$
  \beta(E\oplus E')=- dz_{1\bar13\bar3}-dz_{2\bar24\bar4} + dz_{12\bar3\bar4} + dz_{\bar1\bar234}.
  $$
Note that this class has $k=\frac23$ and it is $\Spin$-rotable via any
 $$
  c  \in \la dz_{12}+dz_{34} , dz_{13}+dz_{24}, i\,dz_{13}-i\, dz_{24} \ra .
 $$
Let $F=\End(E\oplus E')$ be the $\SU$-bundle associated 
to $E\oplus E'$.

Now take $E''=\cO(H_1)\oplus \cO(H_2)$, where $H_1,H_2$ are the divisors
corresponding to $\omega_1,\omega_2$, respectively. It is polystable w.r.t.\
$\omega=\omega_1+\omega_2$ (corresponding to the polarisation $H=H_1+H_2$), and it has
 $$
 4\beta(E'')=-(H_1-H_2)^2= 2 dz_{1\bar13\bar3}+2dz_{2\bar24\bar4} 
  - 2dz_{1\bar12\bar2} - 2dz_{1\bar14\bar4} - 2dz_{2\bar23\bar3} - 2dz_{3\bar34\bar4}.
  $$
This has $k=\frac13$, and it is $\Spin$-rotable via any 
 $$
c  \in \la dz_{12}+dz_{34}, i\,dz_{12}-i\,dz_{34}, dz_{14}+dz_{23}, i\,dz_{14}-i\,dz_{23}\ra.
 $$
Let $F'=\End E''$.

Consider the sum $F\oplus F'$. This is polystable and it has  
 $$
 \beta=- \frac12 \sum_{i<j} dz_{i\bar{i}j\bar{j}} + dz_{12\bar3\bar4} + dz_{\bar1\bar234}
= \beta_0+\omega^2\, .
  $$
Moreover, it is $\Spin$-rotable with $c =dz_{12}+dz_{34}$.

\subsection{The rotated abelian variety}\label{subsec:7.4}
Now starting with $X$ with reduced period matrix (\ref{eqn:red-per-mat}),
the rotated lattice is
   \begin{equation}\label{eqn:29A}
   \left( \begin{array}{cccc}
 1& \delta r i&  si&  0 \\
 0& -\delta s& r&  \delta i \\   
 -a_1+a_2 ri &  a_2 \delta  +a_1r \delta i &   a_1s i&  a_2s \delta i  \\  
 -a_2s & -a_1s \delta  &  a_1r+ a_2 i&  a_2 r \delta -a_1 \delta  i   
\end{array} \right.  \hspace{3cm}
  \end{equation}
$$ \hspace{3cm}
 \left. \begin{array}{cccc}
   -a_1+a_2 r i&  -a_2 \delta -a_1 r \delta i &  a_1si & -a_2s\delta i \\
   -a_2s &  a_1s \delta  &  a_1r+ a_2 i &  -a_2r \delta + a_1 \delta i\\
  1&  -\delta r i &  s i &  0\\  
 0&  \delta s&  r&  -\delta i  
\end{array} \right).
  $$
The endomorphism 
  \begin{equation}\label{eqn:29B}
 \varphi=\delta \left(\begin{array}{cccc}
  r i & s & 0 & 0 \\ -s & -r i & 0 & 0 \\ 0 & 0 & -r i & -s \\ 0 & 0 &s & r i 
 \end{array} \right)
 \end{equation}
leaves the lattice fixed, and $\varphi^2=-d \, \text{Id}$, so it gives the structure of an $L$-module.
A basis of the lattice, as an $L$-module, is given by the first, third, fifth and seventh vectors, i.e.
 $$
 \widehat\Pi=   \left( \begin{array}{cccc}
 1& si&  -a_1+a_2 r i&   a_1si \\
 0&  r&   -a_2s &  a_1r+ a_2 i \\   
 -a_1+a_2 ri &   a_1s i&  1&  s i \\  
 -a_2s &  a_1r+ a_2 i&    0&   r 
\end{array} \right).
  $$
  
The diagonalizing vectors of $\varphi$ are given by the columns of the matrix
  \begin{equation}\label{eqn:29C}
 P= \frac12 \left(\begin{array}{cccc}
 \sqrt{\frac{r + 1}2}  & 0 &  i\sqrt{\frac{r - 1}2} & 0 \\ 
 \frac{s i}{\sqrt{2(r+1)}} & 0 &- \frac{s}{\sqrt{2(r-1)}} & 0 \\ 
 0 & \sqrt{\frac{r-1}2}  & 0& i \sqrt{\frac{r + 1}2} \\ 
 0 & \frac{s i}{\sqrt{2(r-1)}} & 0 &- \frac{s}{\sqrt{2(r+1)}}
\end{array} \right),
 \end{equation}
where we have taken the vectors to be an orthonormal basis, so that the metric has
still the same (standard) form.

Now we parametrize $r^2+s^2=1$ with a single variable $x\in \RR$, by taking
  \begin{equation}\label{eqn:29D}
   r=\frac{1-x^2}{1+x^2}, \quad s= \frac{2x}{1+x^2}\, .
  \end{equation}
Recall that $s=\delta \hat s$, where $\delta =\sqrt{d}$, so that the change
of variables $x=y/\delta$ yields
 $$
   r=\frac{d-y^2}{d+y^2}, \quad \hat s= \frac{2y}{d+y^2} \, ,
 $$
for $r^2+ d \hat s^2=1$, and we clearly have that $r,\hat s\in \QQ \iff y \in \QQ$.

In the basis where $\varphi$ has a standard form
$\left( \begin{array}{cccc} i\delta & 0 & 0 &0 \\ 0 & i\delta &0 &0\\
0 & 0& -i\delta &0 \\ 0& 0& 0 & -i\delta \end{array} \right)$, the lattice is 
generated by the columns of 
 $$
 P^{-1} \widehat\Pi=  
\frac{1}{\sqrt{1+x^2}} 
 \left( \begin{array}{cccc}
 1 & x i & -\bar a & x\bar a i \\ axi & a &  -x i &1 \\
 -x & -i & x a & -ai \\  \bar a i & \bar a x& -i & x
\end{array} \right).
$$

Now $I=\frac{1}{\sqrt{d}}\varphi$ is the standard complex structure,
and the subspaces $W_\pm$ are the standard ones, $W_+=\la z_1,z_2\ra$,
$W_-=\la z_3,z_4\ra$, with respect to 
the lattice obtained by conjugating the last two rows of the previous
matrix, that is,
 \begin{equation}\label{eqn:29E}
 C=\left( \begin{array}{cccc}
 1 & x i & -\bar a & x\bar a i \\ a xi & a &  -x i &1 \\
 -x & i & x \bar a & \bar a i \\ - a i & ax& i & x
\end{array} \right)
 \end{equation}
(we are ignoring the overall scalar factor).

We want to put this lattice into standard form (that is, to do another
change of basis so that the lattice becomes the canonical lattice in
$(\CC^4,I)$), and see where the
subspaces $W_\pm$ go. For this, we have to multiply by a suitable matrix in
$U(2,2)$. If we modify the lattice, by taking a different basis, this
gets easier. Consider the numbers
$$
  f=\frac{2 i x}{1+x^2}= \frac{2 i\delta y}{d+y^2}\, , \qquad
  g=\frac{1-x^2}{1+x^2}=\frac{d-y^2}{d+y^2}\, ,
  $$
both of which are in $K=\QQ[\sqrt{-d}]$, when $y\in \QQ$.
Consider the new basis for the lattice given by
 $$
 \hat C= C \, \left( \begin{array}{cccc} 
 1 &0 & -f & 0 \\ 0 & 0 & g & 0 \\ 
 0 & 0& 0 & g \\ 0 & 1 & 0 & f
 \end{array} \right)=
  \left( \begin{array}{cccc}
 1 & x\bar a i & - xi& - 2 \bar a  \\
 axi & 1  & a & xi \\
 -x & i \bar a & -i& - x \bar a  \\
 - a i & x & -a x & i  
 \end{array} \right).
$$
As we change the basis by using rational numbers, 
the resulting torus is actually a torus
isogenous to the one we started with.

Now let us restrict to $x \in (0,1)$. We observe that $\hat C \in U(2,2)$.
To make the lattice standard, we just conjugate by (forgetting again an overall factor)
$$
 \hat C^{-1} = 
 \left( \begin{array}{cccc} 
1 & -\bar a x i& -x & -\bar a i \\
-a xi & 1 & - ai & -x \\ a & xi & ax & -i \\
-xi & - \bar a & i & -x \bar a
\end{array} \right).
$$
Note that the column vectors are orthogonal w.r.t.\ the indefinite
metric $\tilde h$.

The columns of the matrix $\hat C^{-1}$ define the subspaces $W_+$ and $W_-$. We
can transform  $\hat C^{-1}$, by multiplication on the right with a matrix in
$U(2,2)$ of the form
$\left(\begin{array}{cc} P & 0 \\ 0 & Q \end{array}\right)$, to the matrix
 \begin{equation}\label{eqn:BT}
 \left( \begin{array}{cc} I & \bar B^T \\ B & I \end{array} \right),
 \end{equation}
where
 $$
 B=\frac{1}{1+x^2|a|^2} \left( \begin{array}{cc} a(1-x^2) & x (1+ |a|^2)i \\
                                -x (1+ |a|^2) i & -\bar a(1-x^2) 
                               \end{array} \right).
 $$
The columns of (\ref{eqn:BT}) also define the subspaces $W_+$ and $W_-$.

The defining matrix (\ref{eqn:BT}) corresponds to an abelian variety of Weil type for 
 \begin{align*}
   \tilde a &= \frac{a(1-x^2)}{1+x^2|a|^2}\, ,\\
   \tilde b &= -\frac{x (1+ |a|^2)}{1+x^2|a|^2} i \, .
 \end{align*}
Recall that Theorem \ref{thm:rotation} says that the class $\beta$ is an algebraic cycle for the
Weil abelian variety $X'$ thus obtained. 

\subsection{Which type of abelian variety is $X'$?} 
Recall that we are taking $y\in \QQ$. This guarantees that $X'$ is an abelian variety.
The endomorphism ring of $X$ is $\End (X)= \End Y_1 \x \End Y_2= K\x K$, 
by Lemma \ref{lem:product-non-trivial}. 
Now we want to compute $\End (X')$.

Doing the change of variables $x=y/\delta$, we get
\begin{equation}\label{eqn:ath}
\begin{aligned}
   \tilde a &= \frac{a(d-y^2)}{d+y^2|a|^2} \, ,\\
   \tilde b &= -\delta \, \frac{y (1+ |a|^2)}{d+y^2|a|^2}i =   \varpi i\, .
 \end{aligned}
\end{equation}
A pair $(\tilde a, \varpi)\in \CC\x \RR$ 
can appear as in (\ref{eqn:ath}) if it satisfies the 
equation
 \begin{equation}\label{eqn:relation}
 q (1+|\tilde a|^2+ \varpi^2)+ \varpi i= 0,
 \end{equation}
where $ q=\frac{y \delta i}{y^2+d}$. Note that $y\in\QQ$ implies that $q\in K$.



Equation (\ref{eqn:relation}) gives a smoothly varying family of 
K\"ahler $4$-tori (of Weil type) for all real values $y\in (0,\delta)$, 
although we only know that $\beta$ is HYM for $y\in \QQ$.
Let $X'$ be the K\"ahler $4$-torus for given $(a,\varpi,y)$ 
satisfying (\ref{eqn:relation}). The groups $\End (X')$ form a family.
In particular, for small $y$ we have that $\dim \End (X')\leq 4$.

Let $\psi$ be an endomorphism of $X'$. 
Suppose that $\psi$ commutes with the action of $\varphi$. 
Then either it preserves or it swaps the 
$\pm i \delta$-eigenspaces. Suppose the first case. 
The argument of Lemma \ref{lem:product-non-trivial}
says that there is some matrix $M\in \GL(4,K)$ such that 
the defining matrix satisfies
 $$
 \left( \begin{array}{cc} I & \bar B^T \\ B & I \end{array}\right)
 =
  \left( \begin{array}{cc} P_1 & 0 \\ 0 & P_4 \end{array} \right)
   \left( \begin{array}{cc} I & \bar B^T \\  B & I \end{array}\right)
  \left( \begin{array}{cc} M_1 & M_2 \\ M_3 & M_4 \end{array}\right).
$$
This gets rewritten as
$$
 (M_3+M_4 B)(M_1+M_2 B)^{-1}=B
$$
or equivalently,
 \begin{equation} \label{eqn:B-equiv}
BM_2 B+ BM_1-M_4 B-M_3 =0.
 \end{equation}
In our situation, we have the equation (\ref{eqn:relation}). For fixed $q$, we take 
$(\tilde a,\varpi)$ as generic as possible within this constraint. More explicitly,
writing $\tilde a=a_1 + ia_2$,
we ask that any two of $a_1,a_2, \varpi$ are algebraically independent, and 
also the three $a_1,a_2,\varpi$ are $K$-independent. An easy calculation shows 
that under these
circunstances, (\ref{eqn:B-equiv}) implies that 
 \begin{equation} \label{eqn:B-equiv-M}
 M= \left( \begin{array}{cc|cc} \alpha & 0 & 0 &\gamma \\ 0 & \beta & \gamma & 0 \\ \hline
 0 &\gamma & \alpha & 0 \\ \gamma & 0 & 0 & \beta \end{array} \right),
 \end{equation}
where $\gamma= q\alpha - q\beta$. 
This produces a vector space over $K$ of dimension $2$. Hence this constitutes the
full endomorphism ring, and it is so for any value of $y$ (even not generic ones).


The endomorphism ring of $X'$ is generated by two elements, $\id$ and
$$
F=\left(\begin{array}{cccc} 1 & 0 & 0 & 2q \\ 0 & -1 & 2q & 0 \\ 0 & 2q & 1 & 0\\
2q & 0 &0 &-1 \end{array}\right).
$$
We have that $F^2= (4q^2+1) \id$. Therefore $\End(X')= \QQ[\sqrt{-d},\sqrt{4q^2+1}]$,
which is an imaginary quadratic extension of a degree $2$ totally real field.
So $X'$ is of type $VI(4,1)$ in the classification of abelian varieties (see Appendix B in \cite{survey}).
For these abelian varieties, all Hodge classes are algebraic, because they are products of divisors.

\subsection{Discussion}
In Section \ref{sec:rotating-Weil} we have given an example of a Weil abelian variety $X$
for which there is a (poly)stable vector bundle $E\to X$ which is $\Spin$-rotable.
This produced another Weil abelian variety $X'$ with another (poly)stable vector bundle
whose second Chern class has a non-zero component Weil class. Therefore, such Weil class is
algebraic. The fact that it is a product of divisors has a geometric explanation: the 
$\Spin(7)$-connection on the bundle is actually decomposable ($E$ splits as a direct sum
of line bundles with $\Spin(7)$-connections), therefore the same must happen for the
(poly)stable bundle on $X'$. 

Nonetheless, the $\Spin$-rotations relate abelian varieties of very different geometric and
arithmetic nature. This can be used to produce stable bundles on abelian varieties
and hence to prove algebraicity of Hodge classes. For obtaining deeper consequences, one should
start with non-decomposable stable bundles. 

Moreover, these $\Spin$-rotations produce also K\"ahler non-algebraic tori (e.g. when $y\notin\QQ$
in the situation of Section \ref{sec:rotating-Weil}). Therefore the construction is useful
to produce stable bundles on K\"ahler tori.


\begin{thebibliography}{99}

\bibitem{Bando-Siu}
Bando, S.; Siu, Y.-T. \textit{Stable sheaves and Einstein-Hermitian metrics}, 
In ``Geometry and Analysis
on Complex Manifolds'' (T. Mabuchi et al., ed.), World Scientific, 1994, pp. 39--50.

\bibitem{Debarre} Debarre, O. 
\textit{The diagonal property for abelian varieties}, In
``International Conference on Curves and Abelian Varieties (Athens, 2007)'', 
(V. Alexeev, A. Beauville, H. Clemens, E. Izadi, eds.),
Contemporary Math. \textbf{465}, AMS, 2008, pp. 45--50.

\bibitem{Donaldson-Thomas} Donaldson, S. K.; Thomas, R. P.
\textit{Gauge theory in higher dimensions.} In ``The geometric universe
(Oxford, 1996)'', Oxford Univ. Press, Oxford, 1998, pp. 31--47.

\bibitem{Joyce} Joyce, D. D. \textit{Compact manifolds with special holonomy}
Oxford Mathematical Monographs series, OUP, 2000. 

\bibitem{Lefschetz} Lefschetz, S. 
\textit{L'Analysis situs et la géométrie algébrique}, Collection de
Monographies publiée sous la Direction de M. Emile Borel, Paris, 
Gauthier-Villars [1924].

\bibitem{Lewis} Lewis, C. \textit{$\Spin(7)$ Instantons}, D. Phil.\ thesis,
Oxford, 1998.

\bibitem{survey} Lewis, J.D. \textit{A survey of the Hodge conjecture}, 
2nd ed., CRM Monogr. Ser. \textbf{10}, AMS, 1999. 

\bibitem{Verbitsky} Moraru, R.; Verbitsky, M.
\textit{Stable bundles on hypercomplex surfaces},
Cent. Eur. J. Math. \textbf{8} (2010) 327--337.

\bibitem{Mistretta} Mistretta, E. \textit{Stable vector bundles as generators 
of the Chow ring},
Geometriae Dedicata \textbf{117} (2006) 203--213.

\bibitem{Moonen-Zarhin} Moonen, B. J. J.; Zarhin, Y. G.
\textit{Hodge classes on abelian varieties of low dimension},
Math. Ann. \textbf{315} (1999) 711--733.

\bibitem{Mumford} Mumford, D. \textit{A Note of Shimura's paper ``Discontinuous 
groups and abelian varieties''}, Math. Ann. \textbf{181} (1969) 345--351.

\bibitem{Ramadas} Ramadas, T. M., \textit{$\Spin(7)$ instantons and the 
Hodge Conjecture for certain abelian four-folds: A modest proposal}.
In ``Vector Bundles and Complex Geometry.
Conference on Vector Bundles in Honor of S. Ramanan on the Occasion
of his 70th Birthday'' (O. Garc\'{\i}a-Prada,
P.E. Newstead, L. \'Alvarez-C\'onsul, I. Biswas,
S.B. Bradlow, T.L. G\'omez, eds.)
Contemporary Math. \textbf{522}, AMS, pp. 155--170.

\bibitem{Reyes} Reyes Carri\'{o}n, R. \textit{A generalization of the notion of instanton},
{Differential Geom. Appl.}  \textbf{8}  (1998) 1--20.

\bibitem{Salamon} Salamon, S. \textit{Riemannian Geometry and Holonomy Groups}.
Longman Sc. \& Tech., 1989.

\bibitem{Schoen} Schoen, C.
\textit{Hodge classes on self-products of a variety with an automorphism}, 
Compositio Math. \textbf{65} (1988) 3--32; 
Addendum, Compositio Math. \textbf{114} (1998) 329--336.

\bibitem{Tian} Tian, G. \textit{Gauge theory and calibrated geometry. I},
{Ann. of Math. (2)}  \textbf{151}  (2000) 193--268.

\bibitem{Toma} Toma, M.
\textit{Stable bundles with small $c_2$ over 2-dimensional complex tori},
Math. Z. \textbf{232} (1999) 511--525.     

\bibitem{Voisin} Voisin, C. \textit{A counterexample to the Hodge conjecture extended 
to K\"{a}hler varieties}, {Int. Math. Res. Not.} \textbf{20} (2002) 1057--1075.

\bibitem{Weil} Weil, A. \textit{Abelian varieties and the Hodge ring}, 
Collected papers III: pp. 421--429 [1977].

\bibitem{Zucker} Zucker, S. \textit{The Hodge Conjecture for cubic fourfolds}, Comp. Math. \textbf{34} (1977) 199--209.

\end{thebibliography}
\end{document}